\documentclass[11pt]{article}
\usepackage{amsmath,fullpage,amssymb,amsthm}

\usepackage{tikz,hyperref}
\usetikzlibrary{calc}

\newtheorem{theorem}{Theorem}
\newtheorem{proposition}[theorem]{Proposition}
\newtheorem{lemma}[theorem]{Lemma}
\newtheorem{corollary}[theorem]{Corollary}

\newtheorem{problem}{Problem}

\newtheorem*{definition}{Definition}

\def\Z{\mathbb{Z}}
\def\A{\mathcal{A}}
\def\G{\mathcal{G}}
\def\P{\mathcal{P}}
\def\AA{\mathcal{A}^{(2)}}
\def\PP{\mathcal{P}^{(2)}}
\def\GG{\mathcal{G}^{(2)}}
\def\D{\mathcal{D}}
\def\O{\mathcal{O}}
\def\Q{\mathcal{Q}}
\def\H{\mathcal{H}}
\def\M{\mathcal{M}^{(2)}}
\def\fn2{\lfloor \frac{n}{2} \rfloor}
\def\cn2{\lceil \frac{n}{2} \rceil}
\DeclareMathOperator{\sh}{sh}
\def\tP{\widetilde{P}}
\def\tQ{\widetilde{Q}}
\def\R{\mathcal{R}}
\def\bij{\varphi}
\def\bijMG{\psi}
\def\abs[#1]{|#1|}
\def\bijw{\tilde\varphi}
\def\bijMGw{\tilde\psi}

\begin{document}

\title{Bijections for pairs of non-crossing lattice paths\\ and walks in the plane}

\author{Sergi Elizalde\\
Department of Mathematics\\
Dartmouth College \\
Hanover, NH 03755 \\
\texttt{sergi.elizalde@dartmouth.edu}}

\date{}

\maketitle

\begin{abstract}
It is a classical result in combinatorics that among lattice paths with $2m$ steps $U=(1,1)$ and $D=(1,-1)$ starting at the origin, the number of those that do not go below the $x$-axis 
equals the number of those that end on the $x$-axis. 
A much more unfamiliar fact is that the analogous equality obtained by replacing single paths with $k$-tuples of non-crossing paths holds for every $k$. This result has appeared in the literature in different contexts involving plane partitions (where it was proved by Proctor), partially ordered sets, Young tableaux, and lattice walks, but no bijective proof for $k\ge2$ seems to be known.

In this paper we give a bijective proof of the equality for $k=2$, showing that for pairs of non-crossing lattice paths with $2m$ steps $U$ and $D$, the number of those that do not go below the $x$-axis equals the number of those that end on the $x$-axis.
Translated in terms of walks in the plane starting at the origin with $2m$ unit steps in the four coordinate directions, our work provides correspondences among those constrained to the first octant, those constrained to the first quadrant that end on the $x$-axis, and those in the upper half-plane that end at the origin. 

Our bijections, which are defined in more generality, also prove new results where different endpoints are allowed, and they give a bijective proof of the formula for the number of walks in the first octant that end on the diagonal, partially answering a question of Bousquet-M\'elou and Mishna.
\end{abstract}


\section{Introduction}

For the purpose of this article, a {\em (lattice) path} is a path in $\Z^2$ with steps $U=(1,1)$ and $D=(1,-1)$ starting at the origin $(0,0)$. The {\em length} of a path is its number of steps, which we will denote by $n$.
Let $\A_n$ be the set of all lattice paths of length $n$, and note that $|\A_n|=2^n$.

A {\em Dyck path} is a lattice path that does not go below the $x$-axis and ends on the $x$-axis. Denote the set of Dyck paths of length $2m$ by $\D_{2m}$. It is well known that $|\D_{2m}|=C_m=\frac{1}{m+1}\binom{2m}{m}$, the $m$th Catalan number.

A {\em Grand Dyck path} of length $n$ is a lattice path that ends at $(n,0)$ (for even $n$) or at $(n,1)$ (for odd $n$). Denote the set of Grand Dyck paths (sometimes called {\em free} Dyck paths) of length $n$ by $\G_n$.
It is easy to see that $|\G_n|=\binom{n}{\fn2}$, since constructing a Grand Dyck path is equivalent to choosing which $\fn2$ among the $n$ steps of the path are down-steps.

A {\em Dyck path prefix} is a lattice path that does not go below the $x$-axis, but can end at any height. 
Denote the set of Dyck path prefixes (sometimes called {\em ballot} paths) of length $n$ by $\P_n$.
Note that, by definition, $\P_{2m}\cap\G_{2m}=\D_{2m}$.
Counting Dyck path prefixes is a not as straightforward as counting Grand Dyck paths, but there are several ways to show that $|\P_n|=\binom{n}{\fn2}$. One such way is to provide a bijection between $\P_n$ and $\G_n$. 
Next we describe two known bijections between these sets.

The first one, which we denote by $\xi$, belongs to mathematical folklore and has been used in slightly different forms in \cite{GreKle,EliRub,BBES}. The crucial idea is the construction of a matching between $U$s and $D$s that face each other in the path, in the sense that their midpoints are at the same height and the horizontal line segment (called a {\em tunnel} in~\cite{Eli}) joining them stays below the path. Thinking of the $U$s as opening parentheses and the $D$s as closing parentheses, the matched parentheses properly close each other. Such a matching exists for every lattice path, and it is unique, although in general not all the steps are matched. This matching will play an important role in our bijections in Section~\ref{sec:bij}. Note that among the unmatched steps of the path, the $D$ steps are always to the left of the $U$ steps. Otherwise, if a $U$ came before a $D$, then the higher one of these two steps (or both if they are at the same height) would have been matched.

Given $P\in\P_n$, in order to define $\xi(P)$, we start by matching $U$s and $D$s that face each other in $P$ as described above. Figure~\ref{fig:xi} shows an example. Since $P\in\P_n$, all $D$ steps are matched, and so the only possibly unmatched steps are $U$ steps. Let $j$ be the number of unmatched steps, which also equals the ending height ($y$-coordinate) of $P$, and note that $j$ and $n$ have the same parity. Let $\xi(P)$ be the path obtained by changing the leftmost $\lfloor\frac{j}{2}\rfloor$ unmatched $U$ steps of $P$ into $D$ steps.

It is clear that $\xi(P)\in\G_n$, since this path has $\fn2$ $D$ steps and $\cn2$ $U$ steps. Note also that the pairs of steps that face each other in $P$ are precisely the same pairs of steps that face each other in $\xi(P)$. This observation allows us to find the inverse map, showing that $\xi$ is a bijection. Indeed, given $Q\in\G_n$, we again start by matching $U$s and $D$s that face each other in $Q$, and note that the unmatched $D$s precede the unmatched $U$s.
Changing all the unmatched $D$s into $U$s  we obtain $\xi^{-1}(Q)$.

\begin{figure}[htb]
  \begin{center}
    \begin{tikzpicture}[scale=0.5]
    \def\U{(1,1)}     \def\D{(1,-1)}
      \draw (0,0) coordinate(d0)
      -- ++\U coordinate(d1)
      -- ++\U coordinate(d2)
      -- ++\D coordinate(d3)
      -- ++\D coordinate(d4)
      -- ++\U coordinate(d5)
      -- ++\U coordinate(d6)
      -- ++\D  coordinate(d7)
      -- ++\U coordinate(d8)
      -- ++\U coordinate(d9)
      -- ++\D  coordinate(d10)
      -- ++\D  coordinate(d11)
      -- ++\U coordinate(d12)
      -- ++\U coordinate(d13)
      -- ++\U coordinate(d14)
      -- ++\D  coordinate(d15)
      -- ++\U coordinate(d16);
      \draw[densely dotted] (0,0)--(16,0);
      \draw[dotted,orange] (.5,.5)--(3.5,.5);
      \draw[dotted,orange] (1.5,1.5)--(2.5,1.5);
      \draw[dotted,orange] (5.5,1.5)--(6.5,1.5);
      \draw[dotted,orange] (7.5,1.5)--(10.5,1.5);
      \draw[dotted,orange] (8.5,2.5)--(9.5,2.5);
      \draw[dotted,orange] (13.5,3.5)--(14.5,3.5);
      \draw[green,ultra thick] (d4)--(d5);
      \draw[green,ultra thick] (d11)--(d12);
      \draw[purple,thick] (d12)--(d13);
      \draw[purple,thick] (d15)--(d16);
      \foreach \x in {0,...,16} {
        \draw (d\x) circle (1.5pt);
      }
      \draw (8,-1.5) node [label=$\xi$,rotate=-90] {$\mapsto$};
       \draw (0,-5) coordinate(d0)
      -- ++\U coordinate(d1)
      -- ++\U coordinate(d2)
      -- ++\D coordinate(d3)
      -- ++\D coordinate(d4)
      -- ++\D coordinate(d5) %
      -- ++\U coordinate(d6)
      -- ++\D  coordinate(d7)
      -- ++\U coordinate(d8)
      -- ++\U coordinate(d9)
      -- ++\D  coordinate(d10)
      -- ++\D  coordinate(d11)
      -- ++\D coordinate(d12) %
      -- ++\U coordinate(d13)
      -- ++\U coordinate(d14)
      -- ++\D  coordinate(d15)
      -- ++\U coordinate(d16);
      \draw[densely dotted] (0,-5)--(16,-5);
      \draw[dotted,orange] (.5,-4.5)--(3.5,-4.5);
      \draw[dotted,orange] (1.5,-3.5)--(2.5,-3.5);
      \draw[dotted,orange] (5.5,-5.5)--(6.5,-5.5);
      \draw[dotted,orange] (7.5,-5.5)--(10.5,-5.5);
      \draw[dotted,orange] (8.5,-4.5)--(9.5,-4.5);
      \draw[dotted,orange] (13.5,-5.5)--(14.5,-5.5);
      \draw[green,ultra thick] (d4)--(d5);
      \draw[green,ultra thick] (d11)--(d12);
      \draw[purple,thick] (d12)--(d13);
      \draw[purple,thick] (d15)--(d16);
      \foreach \x in {0,...,16} {
        \draw (d\x) circle (1.5pt);
      }
    \end{tikzpicture}
  \end{center}
  \caption{The bijection $\xi:\P_n\to\G_n$. The unmatched steps changed by $\xi$ are thicker and green.}\label{fig:xi}
\end{figure}
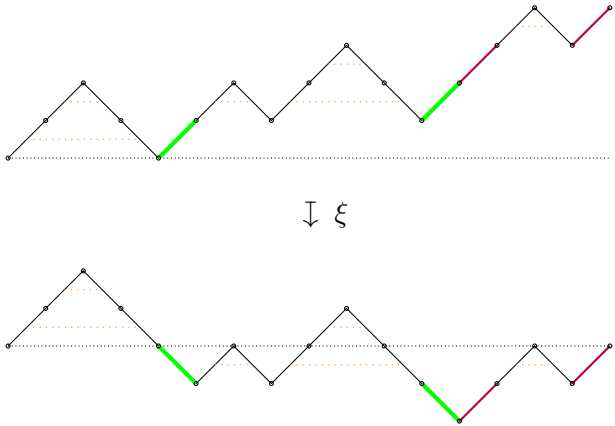

A second bijection $\nu$ between $\P_n$ and $\G_n$, which we will not use in this paper, is due to Nelson~\cite[p.67]{Feller}, and it is described in \cite{Callan}.
Given $P\in\P_n$, let $h$ be its ending height, and let $A$ be the last point of $P$ at height $\lfloor\frac{h}{2}\rfloor$. We construct a Grand Dyck path by splitting $P$ at point $A$, reflecting the right piece along a vertical axis (equivalently, reading the word from right to left and switching $U$s and $D$s) and reattaching it to the left 
of the left piece (see Figure~\ref{fig:nu}). The inverse map is obtained by splitting the Grand Dyck path at its leftmost lowest point, reflecting the left piece, and reattaching it at the right end of the path.

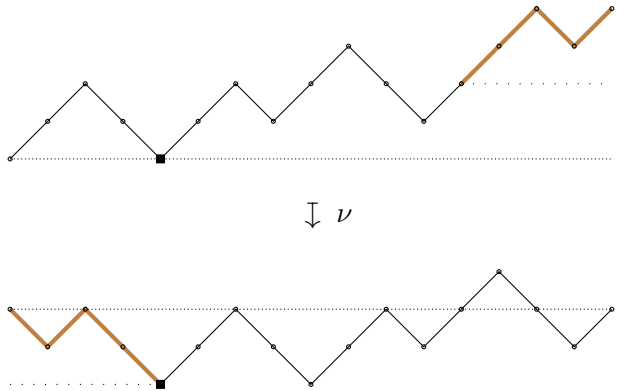
\begin{figure}[htb]
  \begin{center}
    \begin{tikzpicture}[scale=0.5]
    \def\U{(1,1)}     \def\D{(1,-1)}
      \draw (0,0) coordinate(d0)
      -- ++\U coordinate(d1)
      -- ++\U coordinate(d2)
      -- ++\D coordinate(d3)
      -- ++\D coordinate(d4)
      -- ++\U coordinate(d5)
      -- ++\U coordinate(d6)
      -- ++\D  coordinate(d7)
      -- ++\U coordinate(d8)
      -- ++\U coordinate(d9)
      -- ++\D  coordinate(d10)
      -- ++\D  coordinate(d11)
      -- ++\U coordinate(d12)
      -- ++\U coordinate(d13)
      -- ++\U coordinate(d14)
      -- ++\D  coordinate(d15)
      -- ++\U coordinate(d16);
	  \draw[brown,ultra thick] (d12)--(d13)--(d14)--(d15)--(d16);      
      \foreach \x in {0,...,16} {
        \draw (d\x) circle (1.5pt);
      }
      \draw[densely dotted] (0,0)--(16,0);
      \draw[loosely dotted] (d12)--(16,2);
      \filldraw (d4) ++(-3pt,-3pt) rectangle ++(6pt,6pt) ;
      \draw (8,-1.5) node [label=$\nu$,rotate=-90] {$\mapsto$};
 	  \draw (0,-4) coordinate(d0)
      -- ++\D coordinate(d1)
      -- ++\U coordinate(d2)
      -- ++\D coordinate(d3)
      -- ++\D coordinate(d4)
      -- ++\U coordinate(d5)
      -- ++\U coordinate(d6)
      -- ++\D  coordinate(d7)
      -- ++\D coordinate(d8)
      -- ++\U coordinate(d9)
      -- ++\U  coordinate(d10)
      -- ++\D  coordinate(d11)
      -- ++\U coordinate(d12)
      -- ++\U coordinate(d13)
      -- ++\D coordinate(d14)
      -- ++\D  coordinate(d15)
      -- ++\U coordinate(d16);
      \draw[brown,ultra thick] (d0)--(d1)--(d2)--(d3)--(d4);
      \foreach \x in {0,...,16} {
        \draw (d\x) circle (1.5pt);
      }
      \draw[densely dotted] (0,-4)--(16,-4);
      \draw[loosely dotted] (0,-6)--(d4);
      \filldraw (d4) ++(-3pt,-3pt) rectangle ++(6pt,6pt) ;
    \end{tikzpicture}
  \end{center}
  \caption{The bijection $\nu:\P_n\to\G_n$. The splitting point has a square mark, and the piece of the path that is flipped and moved is brown and thicker.}\label{fig:nu}
\end{figure}

\medskip

We write the steps of a lattice path $P$ of length $n$ as $p_1p_2\dots p_n$, where $p_l\in\{U,D\}$ for all $l$. For $0\le a\le n$, the {\em height} of $P$ at $x=a$, denoted $h_a(P)$, is its $y$-coordinate at that point. We denote by $h(P)=h_n(P)$ the ending height of $P$. We write $P\ge0$ to mean that $h_a(P)\ge0$ for all $a$, that is, $P$ 
does not go below the $x$-axis. We denote by $-P$ the path obtained by reflecting $P$ along the $x$-axis.

In this paper we are interested in pairs $(P,Q)$ of lattice paths of the same length where, at every step, $Q$ is weakly below $P$, that is, $h_a(Q)\le h_a(P)$ for all $a$. We say that $(P,Q)$ is a pair of {\em nested} (or {\em non-crossing}) lattice paths, and 
we write $Q\le P$ to denote that $Q$ is weakly below $P$. It is clear that $\le$ defines a partial order.
More generally, we say that $(P_1,\dots,P_k)$ is a $k$-tuple of nested  (or {non-crossing}) lattice paths if $P_{i+1}\le P_i$ for $1\le i\le k-1$.
Denote by $\A^{(k)}_n$  the set of $k$-tuples of nested lattice paths of length~$n$. Similarly, denote by $\G^{(k)}_n$ (resp. $\P^{(k)}_n$) the set of $k$-tuples of nested Grand Dyck paths (resp. Dyck path prefixes) of length~$n$. Note that $\A^{(1)}_n=\A_n$, $\G^{(1)}_n=\G_n$, and $\P^{(1)}_n=\P_n$.

The cardinality of $\G^{(k)}_n$ can be found 
by applying the Gessel-Viennot method~\cite{GV} to count tuples of non-intersecting paths with given endpoints, or
by relating these tuples of paths to plane partitions~\cite{GOV,MacMahon}, as described in Section~\ref{sec:related}. These methods give the known formulas
\begin{equation}\label{eq:Gk}
|\G^{(k)}_n|=\det\left(\binom{n}{\fn2-i+j}\right)_{i,j=1}^k
=\prod_{i=1}^{\cn2}\prod_{j=1}^{\fn2}\prod_{l=1}^k \frac{i+j+l-1}{i+j+l-2}.
\end{equation}

Enumerating $\P^{(k)}_n$ is significantly harder. As described in Section~\ref{sec:related}, sophisticated 
representation-theoretic arguments can be used to show that 
\begin{equation}\label{eq:PG}
|\P^{(k)}_n|=|\G^{(k)}_n|
\end{equation}
for every $k\ge1$. Aside from the case $k=1$ described above, no bijective proof of this equality seems to be known.
Equation~\eqref{eq:PG} may be surprising considering that $(\P_n,\le)$ and $(\G_n,\le)$ are not isomorphic as partially ordered sets (already for $n=4$), that is, there is no bijection between $\P_n$ and $\G_n$ that respects the order relation  of a path being weakly below another. 

In this paper we present a natural bijection between $\PP_n$ and $\GG_n$, described in Section~\ref{sec:bij} in terms of paths.
This map is reminiscent of the bijection $\xi$ between  $\P_n$ and $\G_n$ described above. Our bijection is presented in more generality, allowing different endpoints for the paths.
In Section~\ref{sec:walks} we translate the bijection in terms of walks in the plane with unit steps in the four coordinate directions, constrained to lie in different regions and with restricted endpoints. 
In Section~\ref{sec:related} we discus related work in the literature involving plane partitions, partially ordered sets, tableaux, paths, and walks. We use our set up to provide 
bijective proofs of some known results. 
Section~\ref{sec:proofs} contains the proof of the fact that main map defined in Section~\ref{sec:bij} is bijective.
Finally, we mention a few open problems and future directions in Section~\ref{sec:open}.

\section{The bijection for pairs of nested paths}\label{sec:bij}
The goal of this section is to describe a bijection between $\PP_n$ and $\GG_n$.
Our construction passes through an intermediate set 
$$\M_n=\{(P,Q)\in\AA_n:-P\le Q\le P,\,h(P)=h(Q)\}.$$
We will construct two bijections as follows:
$$\PP_n \overset{\bij}{\longleftarrow} \M_n \overset{\bijMG}{\longrightarrow} \GG_n.$$
Whereas the bijection $\bijMG$ between $\M_n$ and $\GG_n$ will be relatively straightforward, the bijection $\bij$ between $\M_n$ and $\PP_n$ requires more work.

\subsection{The bijection $\bij:\M_n\to\PP_n$}
We can express these two sets as disjoint unions $\M_n=\bigcup_{i} \M_{n,i}$ and $\PP_{n}=\bigcup_{i} \PP_{n,i}$, where
\begin{align*}
\M_{n,i}&=\{(P,Q)\in\AA_n:-P\le Q\le P, \,h(P)=h(Q)=i\},\\
\PP_{n,i}&=\{(P,Q)\in\PP_n:h(Q)=i\},
\end{align*}
and the unions are over all $i$ with $0\le i\le n$ and $i\equiv n\pmod 2$.

We will provide a bijection between $\M_{n,i}$ and $\PP_{n,i}$, and more generally, between the following two sets:
\begin{align*}
\M_{n,i;j}&=\{(P,Q)\in\AA_n:-P\le Q\le P, \,h(P)=i+j,\, h(Q)=i-j\},\\
\PP_{n,i;j}&=\{(P,Q)\in\PP_n:i-j\le h(Q)\le i+j\le h(P)\},
\end{align*}
for any $i\ge j\ge 0$ with $i+j\le n$ and $i+j\equiv n \pmod 2$.
Note that, by definition, $\M_{n,i;0}=\M_{n,i}$ and $\PP_{n,i;0}=\PP_{n,i}$. Thus, our bijection  $\bij:\M_{n,i;j}\to\PP_{n,i;j}$, in the case where $j=0$ and $i$ is allowed to vary, will provide a bijection between $\M_n$ and $\PP_n$.

Before describing $\bij$, let us introduce some terminology. For any $P,Q\in\A_n$, we define their {\em disagreement path} $(P-Q)/2$ to be the path with steps $U$, $D$ and $H=(1,0)$ whose height at each point is half of the difference of heights of $P$ and $Q$. Note that a $U$ step in $(P-Q)/2$ comes from a $U$ step in $P$ and a $D$ step in $Q$, that a $D$ step in $(P-Q)/2$ comes from a $D$ step in $P$ and a $U$ step in $Q$, and that $H$ steps of $(P-Q)/2$ correspond to steps where $P$ and $Q$ agree.

For paths in $\A_n$, we described a matching of $U$ and $D$ steps in the introduction. A very similar matching can be performed for paths with $U$, $D$ and $H$ steps, by just ignoring the $H$ steps, and matching $U$ and $D$ steps that face each other. We define the {\em unmatched} steps of such a path to be the $U$ and $D$ steps that do not get matched in this process. As before, the unmatched $D$ steps are always to the left of the unmatched $U$ steps.

We will use the term {\em flipping} a step to mean changing it from a $U$ to a $D$ or viceversa (equivalently, reflecting it with respect to a horizontal line).

\begin{definition}
Given $(P,Q)\in\M_{n,i;j}$, define $\bij(P,Q)$ to be the pair of paths $(\tP,\tQ)$ constructed as follows:
\begin{enumerate}
\item Let $Q'$ be the path obtained by flipping the steps of $Q$ that end strictly below the $x$-axis.
\item Let $\chi$ be the set of positions of the unmatched $D$ steps of $(P-Q')/2$. 
Let $\tP$ and $\tQ$ be the paths obtained by flipping the steps in $\chi$ of $P$ and $Q'$, respectively.
\end{enumerate}
\end{definition}

An example of the construction of $\bij(P,Q)$ is given in Figure~\ref{fig:bij}.

\begin{figure}[!ht]
  \begin{center}
    \begin{tikzpicture}[scale=0.45]
    \def\U{(1,1)}     \def\D{(1,-1)}
      \draw[densely dotted] (0,0)--(21,0);
      \draw[red,thick] (0,0.05) coordinate(d0)
      -- ++\U coordinate(d1)
      -- ++\U coordinate(d2)
      -- ++\D coordinate(d3)
      -- ++\U coordinate(d4)
      -- ++\U coordinate(d5)
      -- ++\U coordinate(d6)
      -- ++\U  coordinate(d7)
      -- ++\D coordinate(d8)
      -- ++\U coordinate(d9)
      -- ++\U  coordinate(d10)
      -- ++\U  coordinate(d11)
      -- ++\D coordinate(d12)
      -- ++\D coordinate(d13)
      -- ++\U coordinate(d14)
      -- ++\D  coordinate(d15)
      -- ++\D coordinate(d16)      
      -- ++\D coordinate(d17)
      -- ++\U coordinate(d18)
      -- ++\D  coordinate(d19)
      -- ++\U coordinate(d20)      
      -- ++\U coordinate(d21)
      ;
      \foreach \x in {0,...,21} {
        \draw (d\x) circle (1.5pt);
      }
      \draw[blue,thick] (0,-0.05) coordinate(d0)
      -- ++\D coordinate(d1)
      -- ++\U coordinate(d2)
      -- ++\D coordinate(d3)
      -- ++\D coordinate(d4)
      -- ++\U coordinate(d5)
      -- ++\U coordinate(d6)
      -- ++\U  coordinate(d7)
      -- ++\U coordinate(d8)
      -- ++\U coordinate(d9)
      -- ++\U  coordinate(d10)
      -- ++\U  coordinate(d11)
      -- ++\D coordinate(d12)
      -- ++\U coordinate(d13)
      -- ++\U coordinate(d14)
      -- ++\D  coordinate(d15)
      -- ++\D coordinate(d16)      
      -- ++\D coordinate(d17)
      -- ++\D coordinate(d18)
      -- ++\U  coordinate(d19)
      -- ++\D coordinate(d20)      
      -- ++\U coordinate(d21)
      ;
      \foreach \x in {0,...,21} {
        \draw (d\x) circle (1.5pt);
      }
      \draw[red] (6,5) node {$P$};
      \draw[blue] (9,2) node {$Q$};
      \draw[thick,dotted,yellow] (d0)--(d1);
      \draw[thick,dotted,yellow] (d2)--(d3)--(d4)--(d5);
      \draw[thick] (d1)--(d2);
      \draw[thick] (d5)--(d6);
      \draw (1.75,-.75) node {\footnotesize $\R$};
      \draw (5.75,-.75) node {\footnotesize $\R$};
      \draw (21,3) node[right] {\small $i-j$};
      \draw (21,5) node[right] {\small $i+j$};
	  \draw (1,6) node {$\M_{n,i;j}$};
      \draw (11,-1.5) node {$\updownarrow$};
\end{tikzpicture}\\
    \begin{tikzpicture}[scale=0.45]
    \def\U{(1,1)}     \def\D{(1,-1)}    \def\H{(1,0)}
      \draw[densely dotted] (0,0)--(21,0);
      \draw[red,thick] (0,0) coordinate(d0)
      -- ++\U coordinate(d1)
      -- ++\U coordinate(d2)
      -- ++\D coordinate(d3)
      -- ++\U coordinate(d4)
      -- ++\U coordinate(d5)
      -- ++\U coordinate(d6)
      -- ++\U  coordinate(d7)
      -- ++\D coordinate(d8)
      -- ++\U coordinate(d9)
      -- ++\U  coordinate(d10)
      -- ++\U  coordinate(d11)
      -- ++\D coordinate(d12)
      -- ++\D coordinate(d13)
      -- ++\U coordinate(d14)
      -- ++\D  coordinate(d15)
      -- ++\D coordinate(d16)      
      -- ++\D coordinate(d17)
      -- ++\U coordinate(d18)
      -- ++\D  coordinate(d19)
      -- ++\U coordinate(d20)      
      -- ++\U coordinate(d21)
      ;
      \foreach \x in {0,...,21} {
        \draw (d\x) circle (1.5pt);
      }
      \draw[blue,thick] (0,0.05) coordinate(d0)
      -- ++\U coordinate(d1)
      -- ++\U coordinate(d2)
      -- ++\U coordinate(d3)
      -- ++\U coordinate(d4)
      -- ++\D coordinate(d5)
      -- ++\U coordinate(d6)
      -- ++\U  coordinate(d7)
      -- ++\U coordinate(d8)
      -- ++\U coordinate(d9)
      -- ++\U  coordinate(d10)
      -- ++\U  coordinate(d11)
      -- ++\D coordinate(d12)
      -- ++\U coordinate(d13)
      -- ++\U coordinate(d14)
      -- ++\D  coordinate(d15)
      -- ++\D coordinate(d16)      
      -- ++\D coordinate(d17)
      -- ++\D coordinate(d18)
      -- ++\U  coordinate(d19)
      -- ++\D coordinate(d20)      
      -- ++\U coordinate(d21)
      ;
      \foreach \x in {0,...,21} {
        \draw (d\x) circle (1.5pt);
      }
      \draw[red] (10,5) node {$P$};
      \draw[blue] (9,8) node {$Q'$};
      \draw[thick,dotted,yellow] (d0)--(d1);
      \draw[thick,dotted,yellow] (d2)--(d3)--(d4)--(d5);
      \draw[thick] (d1)--(d2);
      \draw[thick] (d5)--(d6);
      \draw (1.75,1.25) node {\footnotesize $\R$};
      \draw (5.75,3.25) node {\footnotesize $\R$};
      \draw[blue,dotted] (.2,.95) node {\footnotesize $Q'_0$};
      \draw[blue,dotted] (3.75,3) node {\footnotesize $Q'_1$};
	 \draw[dashed] (0,0) coordinate(d0)
      -- ++\H coordinate(d1)
      -- ++\H coordinate(d2)
      -- ++\D coordinate(d3)
      -- ++\H coordinate(d4)
      -- ++\U coordinate(d5)
      -- ++\H coordinate(d6)
      -- ++\H  coordinate(d7)
      -- ++\D coordinate(d8)
      -- ++\H coordinate(d9)
      -- ++\H  coordinate(d10)
      -- ++\H  coordinate(d11)
      -- ++\H coordinate(d12)
      -- ++\D coordinate(d13)
      -- ++\H coordinate(d14)
      -- ++\H  coordinate(d15)
      -- ++\H coordinate(d16)      
      -- ++\H coordinate(d17)
      -- ++\U coordinate(d18)
      -- ++\D  coordinate(d19)
      -- ++\U coordinate(d20)      
      -- ++\H coordinate(d21)
      ;
      \foreach \x in {0,...,21} {
        \draw (d\x) circle (1.2pt);
      }
      \draw[thick,green] (d2)--(d3);      \draw[thick,green] (d12)--(d13);
      \draw[green] (2.25,-0.75) node {\footnotesize $\chi$};
      \draw[green] (12.25,-1.75) node {\footnotesize $\chi$};
      \draw (14.5,-1.25) node {$\frac{P-Q'}{2}$};
  	  \draw (21,3) circle (1.5pt);
      \draw (21,3) node[right] {\small $i-j$};
      \draw (21,5) node[right] {\small $i+j$}; 
      \draw (11,-3) node {$\updownarrow$};
\end{tikzpicture}\\
\begin{tikzpicture}[scale=0.45]
  \def\U{(1,1)}     \def\D{(1,-1)}	 \def\H{(1,0)}
     \draw[densely dotted] (0,0)--(21,0);
      \draw[red,thick] (0,0.05) coordinate(d0)
      -- ++\U coordinate(d1)
      -- ++\U coordinate(d2)
      -- ++\U coordinate(d3)
      -- ++\U coordinate(d4)
      -- ++\U coordinate(d5)
      -- ++\U coordinate(d6)
      -- ++\U  coordinate(d7)
      -- ++\D coordinate(d8)
      -- ++\U coordinate(d9)
      -- ++\U  coordinate(d10)
      -- ++\U  coordinate(d11)
      -- ++\D coordinate(d12)
      -- ++\U coordinate(d13)
      -- ++\U coordinate(d14)
      -- ++\D  coordinate(d15)
      -- ++\D coordinate(d16)      
      -- ++\D coordinate(d17)
      -- ++\U coordinate(d18)
      -- ++\D  coordinate(d19)
      -- ++\U coordinate(d20)      
      -- ++\U coordinate(d21)
      ;
      \foreach \x in {0,...,21} {
        \draw (d\x) circle (1.5pt);
      }
      \draw[blue,thick] (0,0) coordinate(d0)
      -- ++\U coordinate(d1)
      -- ++\U coordinate(d2)
      -- ++\D coordinate(d3)
      -- ++\U coordinate(d4)
      -- ++\D coordinate(d5)
      -- ++\U coordinate(d6)
      -- ++\U  coordinate(d7)
      -- ++\U coordinate(d8)
      -- ++\U coordinate(d9)
      -- ++\U  coordinate(d10)
      -- ++\U  coordinate(d11)
      -- ++\D coordinate(d12)
      -- ++\D coordinate(d13)
      -- ++\U coordinate(d14)
      -- ++\D  coordinate(d15)
      -- ++\D coordinate(d16)      
      -- ++\D coordinate(d17)
      -- ++\D coordinate(d18)
      -- ++\U  coordinate(d19)
      -- ++\D coordinate(d20)      
      -- ++\U coordinate(d21)
      ;
      \foreach \x in {0,...,21} {
        \draw (d\x) circle (1.5pt);
      }
      \draw[red] (6,7) node {$\tP$};
      \draw[blue] (9,4) node {$\tQ$};  
	 \draw[dashed] (0,0) coordinate(d0)
      -- ++\H coordinate(d1)
      -- ++\H coordinate(d2)
      -- ++\U coordinate(d3)
      -- ++\H coordinate(d4)
      -- ++\U coordinate(d5)
      -- ++\H coordinate(d6)
      -- ++\H  coordinate(d7)
      -- ++\D coordinate(d8)
      -- ++\H coordinate(d9)
      -- ++\H  coordinate(d10)
      -- ++\H  coordinate(d11)
      -- ++\H coordinate(d12)
      -- ++\U coordinate(d13)
      -- ++\H coordinate(d14)
      -- ++\H  coordinate(d15)
      -- ++\H coordinate(d16)      
      -- ++\H coordinate(d17)
      -- ++\U coordinate(d18)
      -- ++\D  coordinate(d19)
      -- ++\U coordinate(d20)      
      -- ++\H coordinate(d21)
      ;
      \foreach \x in {0,...,21} {
        \draw (d\x) circle (1.2pt);
      }
      \draw[thick, green] (d2)--(d3);      \draw[thick, green] (d12)--(d13);
      \draw[green] (2.25,0.75) node {\footnotesize $\chi$};
      \draw[green] (12.25,1.75) node {\footnotesize $\chi$};
      \draw (14.5,1.25) node {$\frac{\tP-\tQ}{2}$};      
	  \draw (21,5) circle (1.5pt);
      \draw (21,3) node[right] {\small $i-j$};
      \draw (21,5) node[right] {\small $i+j$}; 
      \draw (1,8) node {$\PP_{n,i;j}$};
\end{tikzpicture}
  \end{center}
  \caption{The bijection $\bij:\M_{n,i;j}\to\PP_{n,i;j}$.  The steps that are flipped in going between $Q$ and $Q'$ are dotted in yellow. The steps $\chi$ that are flipped in going between $(P,Q')$ and $(\tP,\tQ)$ are thicker and green in the disagreement path.}\label{fig:bij}
\end{figure}
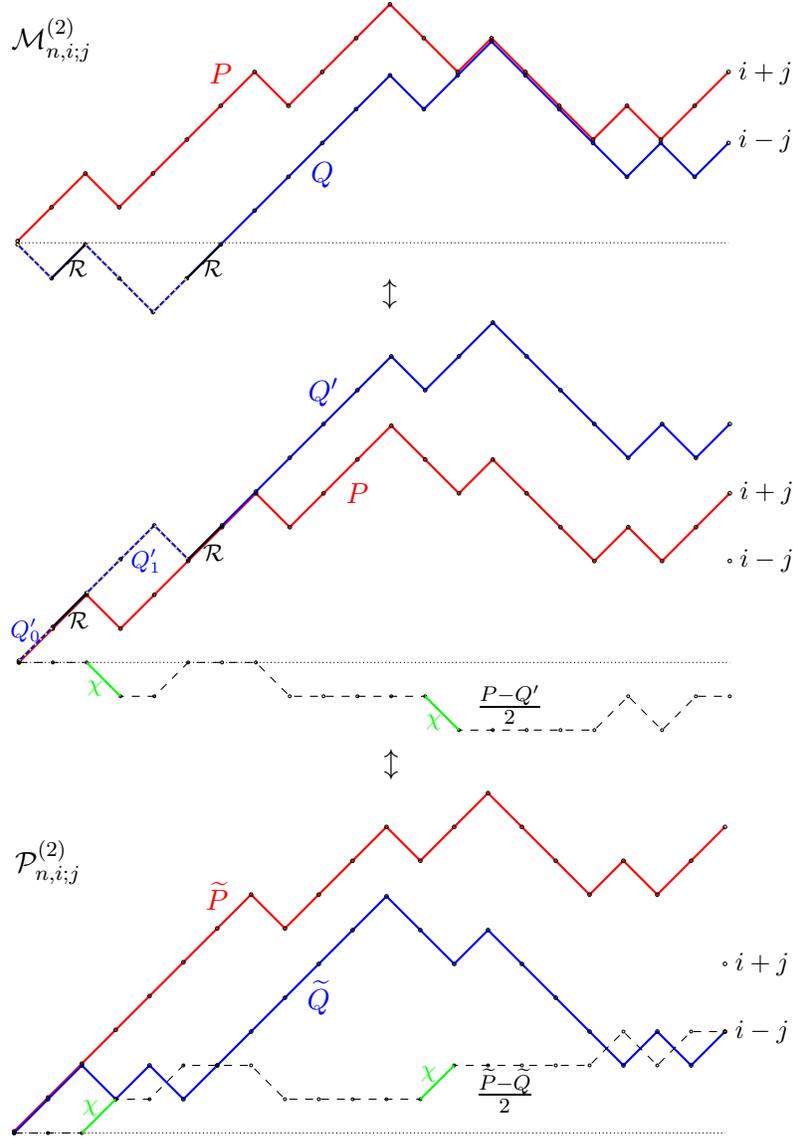

\begin{theorem}\label{thm:bijMP}\label{THM:BIJMP}
The map $\bij$ is a bijection between $\M_{n,i;j}$ and $\PP_{n,i;j}$.
\end{theorem}

For the sake of continuity, the proof of this theorem will be postponed until Section~\ref{sec:proofs}.

\subsection{The bijection $\bijMG:\M_n\to\GG_n$}
As we did for $\bij$ in the previous subsection, we will provide a bijection between more general sets than $\M_n$ and $\GG_n$.
For any $P,Q\in\A_n$, define their {\em agreement path} $(P+Q)/2$ to be the path with steps $U$, $D$ and $H$ whose height at each point is half of the sum of heights of $P$ and $Q$. Note that this path agrees with $P$ and $Q$ in the positions where $P$ and $Q$ agree, and it has $H$ steps in the positions where $P$ and $Q$ disagree. Define $\ell(P,Q)$ to be the $y$-coordinate of the lowest point of $(P+Q)/2$.

Fix $i\ge j\ge 0$ with $i+j\le n$ and $i+j\equiv n \pmod 2$, and let
$$\GG_{n,i;j}=\{(P,Q)\in\AA_n:\ell(P,Q)=\lfloor i/2 \rfloor,\,h(P)=j+\delta_{i},\,h(Q)=-j+\delta_{i}\},$$
where $\delta_i$ denotes the remainder of the division of $i$ by $2$.

Next we describe a bijection $\bijMG:\M_{n,i;j}\to\GG_{n,i;j}$. Note that, in the special case that $j=0$, taking the union over all $i$ with $0\le i\le n$ and $i\equiv n \pmod 2$,
the map $\bijMG$ will give a bijection between $$\M_n=\bigcup_{i} \M_{n,i;0} \quad\text{and}\quad \GG_{n}=\bigcup_{i} \GG_{n,i;0}.$$ 

Given $(P,Q)\in\M_{n,i;j}$, the path $(P+Q)/2$ never goes below the $x$-axis because $-P\le Q$.
The non-horizontal steps in $(P+Q)/2$, which are those in the positions where $P$ and $Q$ agree, form a Dyck path prefix ending at height $$h\left(\frac{P+Q}{2}\right)=\frac{h(P)+h(Q)}{2}=\frac{(i+j)+(i-j)}{2}=i.$$ One can apply to these steps the bijection $\xi$ described in the introduction, which turns the leftmost $\lfloor i/2 \rfloor$ unmatched $U$ steps into $D$ steps, producing a Grand Dyck path. Let $\widehat{P}$ and $\widehat{Q}$ be the paths obtained by changing the corresponding 
$\lfloor i/2 \rfloor$ $U$ steps of $P$ and $Q$, respectively, into $D$ steps. Define $\bijMG(P,Q)=(\widehat{P},\widehat{Q})$.

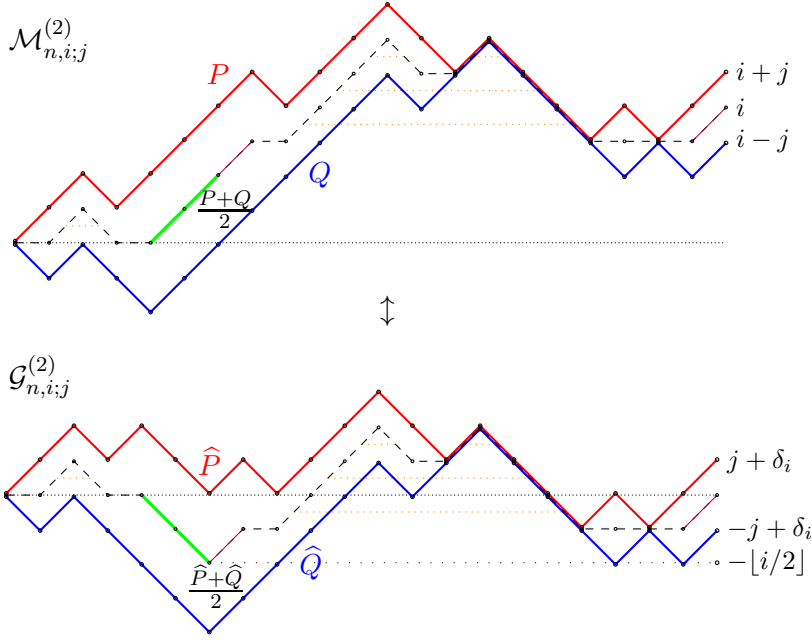
\begin{figure}[htb]
  \begin{center}
  \begin{tabular}{l}
    \begin{tikzpicture}[scale=0.45]
    \def\U{(1,1)}     \def\D{(1,-1)} \def\H{(1,0)}
      \draw[densely dotted] (0,0)--(21,0);
      \draw[red,thick] (0,0.05) coordinate(d0)
      -- ++\U coordinate(d1)
      -- ++\U coordinate(d2)
      -- ++\D coordinate(d3)
      -- ++\U coordinate(d4)
      -- ++\U coordinate(d5)
      -- ++\U coordinate(d6)
      -- ++\U  coordinate(d7)
      -- ++\D coordinate(d8)
      -- ++\U coordinate(d9)
      -- ++\U  coordinate(d10)
      -- ++\U  coordinate(d11)
      -- ++\D coordinate(d12)
      -- ++\D coordinate(d13)
      -- ++\U coordinate(d14)
      -- ++\D  coordinate(d15)
      -- ++\D coordinate(d16)      
      -- ++\D coordinate(d17)
      -- ++\U coordinate(d18)
      -- ++\D  coordinate(d19)
      -- ++\U coordinate(d20)      
      -- ++\U coordinate(d21)
      ;
      \draw[dotted,orange] (1.5,.5)--(2.5,.5);
      \draw[dotted,orange] (8.5,3.5)--(16.5,3.5);
      \draw[dotted,orange] (9.5,4.5)--(15.5,4.5);
      \draw[dotted,orange] (10.5,5.5)--(11.5,5.5);
      \draw[dotted,orange] (13.5,5.5)--(14.5,5.5);
      \foreach \x in {0,...,21} {
        \draw (d\x) circle (1.5pt);
      }      
      \draw[blue,thick] (0,-0.05) coordinate(d0)
      -- ++\D coordinate(d1)
      -- ++\U coordinate(d2)
      -- ++\D coordinate(d3)
      -- ++\D coordinate(d4)
      -- ++\U coordinate(d5)
      -- ++\U coordinate(d6)
      -- ++\U  coordinate(d7)
      -- ++\U coordinate(d8)
      -- ++\U coordinate(d9)
      -- ++\U  coordinate(d10)
      -- ++\U  coordinate(d11)
      -- ++\D coordinate(d12)
      -- ++\U coordinate(d13)
      -- ++\U coordinate(d14)
      -- ++\D  coordinate(d15)
      -- ++\D coordinate(d16)      
      -- ++\D coordinate(d17)
      -- ++\D coordinate(d18)
      -- ++\U  coordinate(d19)
      -- ++\D coordinate(d20)      
      -- ++\U coordinate(d21)
      ;
      \foreach \x in {0,...,21} {
        \draw (d\x) circle (1.5pt);
      }
      \draw[red] (6,5) node {$P$};
      \draw[blue] (9,2) node {$Q$};
      \draw (21,3) node[right] {\small $i-j$};
      \draw (21,4) node[right] {\small $i$};
      \draw (21,5) node[right] {\small $i+j$};
	  \draw (1,6) node {$\M_{n,i;j}$};
	 \draw[dashed] (0,0) coordinate(d0)
      -- ++\H coordinate(d1)
      -- ++\U coordinate(d2)
      -- ++\D coordinate(d3)
      -- ++\H coordinate(d4)
      -- ++\U coordinate(d5)
      -- ++\U coordinate(d6)
      -- ++\U  coordinate(d7)
      -- ++\H coordinate(d8)
      -- ++\U coordinate(d9)
      -- ++\U  coordinate(d10)
      -- ++\U  coordinate(d11)
      -- ++\D coordinate(d12)
      -- ++\H coordinate(d13)
      -- ++\U coordinate(d14)
      -- ++\D  coordinate(d15)
      -- ++\D coordinate(d16)      
      -- ++\D coordinate(d17)
      -- ++\H coordinate(d18)
      -- ++\H  coordinate(d19)
      -- ++\H coordinate(d20)      
      -- ++\U coordinate(d21)
      ;
      \draw[green,very thick] (d4)--(d6);      
      \draw[purple] (d6)--(d7);         \draw[purple] (d20)--(d21);
      \foreach \x in {0,...,21} {
        \draw (d\x) circle (1.2pt);
      }
      \draw (6.2,1) node {$\frac{P+Q}{2}$};  
	  
      \draw (11,-2) node {$\updownarrow$};
\end{tikzpicture}\\
   \begin{tikzpicture}[scale=0.45]
    \def\U{(1,1)}     \def\D{(1,-1)} \def\H{(1,0)}
      \draw[densely dotted] (0,0)--(21,0);
      \draw[loosely dotted, thin] (6,-2)--(21,-2);
      \draw[red,thick] (0,0.05) coordinate(d0)
      -- ++\U coordinate(d1)
      -- ++\U coordinate(d2)
      -- ++\D coordinate(d3)
      -- ++\U coordinate(d4)
      -- ++\D coordinate(d5)
      -- ++\D coordinate(d6)
      -- ++\U  coordinate(d7)
      -- ++\D coordinate(d8)
      -- ++\U coordinate(d9)
      -- ++\U  coordinate(d10)
      -- ++\U  coordinate(d11)
      -- ++\D coordinate(d12)
      -- ++\D coordinate(d13)
      -- ++\U coordinate(d14)
      -- ++\D  coordinate(d15)
      -- ++\D coordinate(d16)      
      -- ++\D coordinate(d17)
      -- ++\U coordinate(d18)
      -- ++\D  coordinate(d19)
      -- ++\U coordinate(d20)      
      -- ++\U coordinate(d21)
      ;
      \draw[dotted,orange] (1.5,.5)--(2.5,.5);
      \draw[dotted,orange] (8.5,-.5)--(16.5,-.5);
      \draw[dotted,orange] (9.5,.5)--(15.5,.5);
      \draw[dotted,orange] (10.5,1.5)--(11.5,1.5);
      \draw[dotted,orange] (13.5,1.5)--(14.5,1.5);
      \foreach \x in {0,...,21} {
        \draw (d\x) circle (1.5pt);
      }      
      \draw[blue,thick] (0,-0.05) coordinate(d0)
      -- ++\D coordinate(d1)
      -- ++\U coordinate(d2)
      -- ++\D coordinate(d3)
      -- ++\D coordinate(d4)
      -- ++\D coordinate(d5)
      -- ++\D coordinate(d6)
      -- ++\U  coordinate(d7)
      -- ++\U coordinate(d8)
      -- ++\U coordinate(d9)
      -- ++\U  coordinate(d10)
      -- ++\U  coordinate(d11)
      -- ++\D coordinate(d12)
      -- ++\U coordinate(d13)
      -- ++\U coordinate(d14)
      -- ++\D  coordinate(d15)
      -- ++\D coordinate(d16)      
      -- ++\D coordinate(d17)
      -- ++\D coordinate(d18)
      -- ++\U  coordinate(d19)
      -- ++\D coordinate(d20)      
      -- ++\U coordinate(d21)
      ;
      \foreach \x in {0,...,21} {
        \draw (d\x) circle (1.5pt);
      }
      \draw[red] (6,1) node {$\widehat{P}$};
      \draw[blue] (9,-2) node {$\widehat{Q}$};
      \draw (21,1) node[right] {\small $j+\delta_i$};
      \draw (21,-1) node[right] {\small $-j+\delta_i$};
  	  \draw (21,-2) circle (1.5pt);
  	  \draw (21,-2) node[right] {\small $-\lfloor i/2 \rfloor$};
	  \draw (1,3.5) node {$\GG_{n,i;j}$};
	 \draw[dashed] (0,0) coordinate(d0)
      -- ++\H coordinate(d1)
      -- ++\U coordinate(d2)
      -- ++\D coordinate(d3)
      -- ++\H coordinate(d4)
      -- ++\D coordinate(d5)
      -- ++\D coordinate(d6)
      -- ++\U  coordinate(d7)
      -- ++\H coordinate(d8)
      -- ++\U coordinate(d9)
      -- ++\U  coordinate(d10)
      -- ++\U  coordinate(d11)
      -- ++\D coordinate(d12)
      -- ++\H coordinate(d13)
      -- ++\U coordinate(d14)
      -- ++\D  coordinate(d15)
      -- ++\D coordinate(d16)      
      -- ++\D coordinate(d17)
      -- ++\H coordinate(d18)
      -- ++\H  coordinate(d19)
      -- ++\H coordinate(d20)      
      -- ++\U coordinate(d21)
      ;
      \draw[green,very thick] (d4)--(d6);      
      \draw[purple] (d6)--(d7);         \draw[purple] (d20)--(d21);
      \foreach \x in {0,...,21} {
        \draw (d\x) circle (1.2pt);
      }
      \draw (6.2,-2.7) node {$\frac{\widehat{P}+\widehat{Q}}{2}$};  
\end{tikzpicture}
\end{tabular}
  \end{center}
  \caption{The bijection $\bijMG:\M_{n,i;j}\to\GG_{n,i;j}$. 
  The unmatched steps of the agreement path that are changed are thicker and green, while the remaining unmatched steps are purple.
}
\end{figure}

\begin{proposition}\label{prop:bijMG}
The map $\bijMG$ is a bijection between $\M_{n,i;j}$ and $\GG_{n,i;j}$.
\end{proposition}

\begin{proof}
Let $(P,Q)\in\M_{n,i;j}$ and let $(\widehat{P},\widehat{Q})=\bijMG(P,Q)$.
When $\xi$ is applied to the non-horizontal steps of $(P+Q)/2$ to produce the path $(\widehat{P}+\widehat{Q})/2$,
it changes the leftmost $\lfloor i/2 \rfloor$ unmatched $U$ steps into $D$ steps. It follows that
$$\ell(\widehat{P},\widehat{Q})=-\left\lfloor\frac{i}{2}\right\rfloor,$$
that
$$h(\widehat{P})=h(P)-2\left\lfloor\frac{i}{2}\right\rfloor=i+j-(i-\delta_{i})=j+\delta_{i},$$
and similarly that $h(\widehat{Q})=-j+\delta_{i}$,
so $(\widehat{P},\widehat{Q})\in\GG_{n;j}$.

To see that $\bijMG$ is a bijection, we describe its inverse. Given $(\widehat{P},\widehat{Q})\in\GG_{n,i;j}$, 
consider the path $(\widehat{P}+\widehat{Q})/2$. Its non-horizontal steps determine a Grand Dyck path, since $h((\widehat{P}+\widehat{Q})/2)=\delta_{i}\in\{0,1\}$. 
Applying $\xi^{-1}$ to this path changes its $\lfloor i/2\rfloor$ unmatched $D$ steps into $U$ steps, producing a path that does not go below the $x$-axis. Let $P$ and $Q$ be the paths obtained by changing the corresponding $D$ steps of $\widehat{P}$ and $\widehat{Q}$, respectively, into $U$ steps. 
Then $h(P)=h(\widehat{P})+2 \lfloor i/2\rfloor=i+j$ and $h(Q)=h(\widehat{Q})+2 \lfloor i/2\rfloor=i-j$.
Additionally, since $(P+Q)/2\ge0$, we have that $-P\le Q\le P$. Finally, the fact that $\xi$ is a bijection guarantees that $(P,Q)$ is the unique pair such that $\bijMG(P,Q)=(\widehat{P},\widehat{Q})$.
\end{proof}

\begin{corollary}\label{cor:PG}
The map $\bijMG\circ\bij^{-1}$ restricts to a bijection between $\PP_n$ and $\GG_n$.
\end{corollary}

\begin{proof}
By Theorem~\ref{thm:bijMP} and Proposition~\ref{prop:bijMG}, we have bijections
$$\PP_{n,i;j}\stackrel{\bij^{-1}}{\longrightarrow}\M_{n,i;j}\stackrel{\bijMG}{\longrightarrow}\GG_{n;i,j}.$$
Setting $j=0$ and taking the union over all $i$ with $0\le i\le n$ and $i\equiv n\pmod 2$, we obtain bijections
$$\PP_{n}=\bigcup_i \PP_{n,i;0}\stackrel{\bij^{-1}}{\longrightarrow}\M_{n}=\bigcup_i \M_{n,i;0}\stackrel{\bijMG}{\longrightarrow}\GG_{n}=\bigcup_i \GG_{n,i;0}.$$  
\end{proof}

\section{Walks in the plane}\label{sec:walks}
In this section we interpret the above results about pairs of paths in terms of lattice walks in the plane with steps $N=(0,1)$, $S=(0,-1)$, $E=(1,0)$ and $W=(-1,0)$ starting at the origin. The term {\em walk} will always refer to such a lattice walk in this section. Note that walks are allowed to self-intersect. The {\em length} of a walk is its number of steps.
We write the steps of a walk $w$ as $w_1w_2\dots w_n$, where $w_l\in\{N,S,E,W\}$ for all $l$.

Next we describe a standard bijection $\omega$ between $\A_n\times \A_n$ and the set of walks of length~$n$.
This bijection has been used, among other places, in~\cite{G-B86,BMM}. Given a pair of lattice paths $(P,Q)\in\A_n\times\A_n$, define a walk $w=\omega(P,Q)$ as follows. For each $1\le l\le n$, the $l$-th steps of $P$ and $Q$ determine the $l$-th step of $w$ according to the following rule:
$$\begin{array}{cccc}
p_l&q_l&&w_l\\ \hline
U&U&\mapsto&E\\
U&D&\mapsto&N\\
D&U&\mapsto&S\\
D&D&\mapsto&W
\end{array}$$
Note that if $h_l(P)=a$ and $h_l(Q)=b$, then the coordinate of $\omega(P,Q)$ after $l$ steps is $\left(\frac{a+b}{2},\frac{a-b}{2}\right)$. Equivalently, if the coordinate of $\omega(P,Q)$ after $l$ steps is $(x,y)$, then 
$h_l(P)=x+y$ and $h_l(Q)=x-y$.

An easy consequence is that, under the bijection $\omega$, conditions about the paths $P$ and $Q$ translate into conditions on the walk $\omega(P,Q)$ as described in Table~\ref{tab:omega}.

\begin{table}[h]
\begin{center}
\begin{tabular}{c|c}
Conditions on $(P,Q)\in\A_n\times\A_n$ & Conditions on $w=\omega(P,Q)$ \\ \hline
 $P\ge Q$ & $w$ does not go below the $x$-axis\\
 $Q\ge 0$ & $w$ does not go above the line $y=x$\\
 $-P\le Q$ & $w$ does not go left of the $y$-axis\\
 $h(P)=h(Q)$ & $w$ ends on the $x$-axis\\
 $h(Q)=i$ & $w$ ends on the line $y=x-i$\\
 $h(P)=i+j$ and $h(Q)=i-j$ & $w$ ends at $(i,j)$\\
 $i-j\le h(Q)\le i+j\le h(P)$ & $w$ ends in $\sh(i,j)$
\end{tabular}
\end{center}
\caption{Correspondences between properties of paths and properties of walks through the bijection $\omega$.}
\label{tab:omega}
\end{table}

Let $\O_n$ be the set of walks of length $n$ constrained to the first octant (that is, the region $x\ge y\ge 0$), and let
$\O_{n,i}$ be the subset of those that end on the line $y=x-i$. 
From Table~\ref{tab:omega}, we see that $\omega$ restricts to a bijection between $\PP_{n,i}$ and $\O_{n,i}$ for every $i$,
and thus, taking the union over all $i$, to a bijection between $\PP_{n}$ and $\O_{n}$.

Let $\Q_n$  be the set of walks of length $n$ constrained to the first quadrant (that is, the region $x,y\ge 0$), and let $\Q^x_n$ be the subset of those that end on the $x$-axis. 
Furthermore, let $\Q^x_{n,i}\subset \Q^x_n$ denote the set of paths that end at $(i,0)$.  From Table~\ref{tab:omega}, we see that $\omega$ restricts to a bijection between $\M_{n,i}$ and $\Q^x_{n,i}$ for every $i$,
and thus to a bijection between $\M_{n}$ and $\Q^x_{n}$.

Let $\H_n$ be the set of walks of length $n$ constrained to the upper half-plane (that is, the region $y\ge 0$), and let
$\H^0_n$ be the subset of those that end at $(\delta_n,0)$. Then $\omega$ restricts to a bijection between $\GG_n$ and $\H^0_n$.

Analogously to how we generalized in Section~\ref{sec:bij} the definitions of sets of nested paths by introducing a parameter $j$, the definitions of the above sets of walks can be extended as well.
To generalize $\O_{n,i}$, we need to introduce the concept of {\em shadow} of a point. For any lattice 
point $(i,j)$ with $i\ge j\ge 0$, its shadow is defined to be the region
$$\sh(i,j)=\{(x,y):i-j\le x-y\le i+j\le x+y\}.$$

\begin{figure}[htb]
  \begin{center}
    \begin{tikzpicture}[scale=0.4]
      \draw[fill,orange!20] (9,7)--(5,3)--(8,0)--(11,3);
      \draw[densely dotted] (0,0) -- (11,0);
      \draw[densely dotted] (0,0) -- (8,8);
      \draw[fill] (5,3) circle (2pt);
      \draw (5,3) node[above] {\small $(i,j)$};
      \draw[orange] (9,7)--(5,3)--(8,0)--(11,3);
	  \draw (8,3) node {\small $\sh(i,j)$};
    \end{tikzpicture}
  \end{center}
  \caption{The shadow of a point $(i,j)$.}\label{fig:sh}
\end{figure}
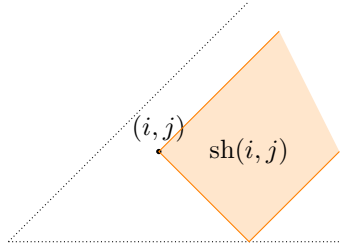

For $i\ge j\ge0$ with $i+j\le n$ and $i+j\equiv n \pmod 2$, define the following sets. Let $\O_{n,i;j}$ be the set of walks in $\O_n$ that end in the region $\sh(i,j)$. 
Let $\Q_{n,i;j}$ be the set of walks in $\Q_n$ that end at the point $(i,j)$. 
Note that $\O_{n,i;0}=\O_{n,i}$ and $\Q_{n,i;0}=\Q^x_{n,i}$.
Let $\H_{n,i;j}$ be the set of walks in $\H_n$ that end at $(\delta_{i},j)$ whose leftmost point lies on the line $x=-\lfloor i/2 \rfloor$.
Note that $\bigcup_i\H_{n,i;0}=\H^0_{n}$, where the union is over all $0\le i\le n$ with $i\equiv n \pmod 2$.

Using Table~\ref{tab:omega}, it is easy to check that $\omega$ gives bijections between $\PP_{n,i;j}$ and $\O_{n,i;j}$, between $\M_{n,i;j}$ and $\Q_{n,i;j}$, and between $\H_{n,i;j}$ and $\GG_{n,i;j}$.
In particular, the bijections in Section~\ref{sec:bij} can be interpreted, via $\omega$, as bijections for lattice walks as shown in Figure~\ref{fig:bijections_ij}, where we write $\bijw:=\omega\circ\bij\circ\omega^{-1}$ and $\bijMGw:=\omega\circ\bijMG\circ\omega^{-1}$.

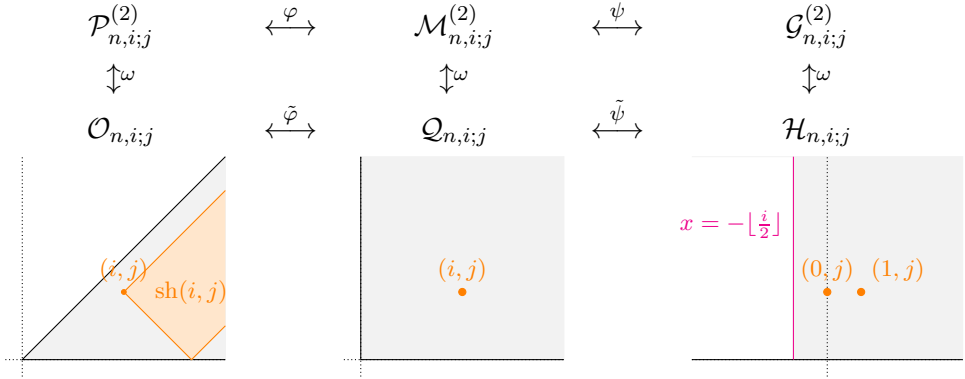
\begin{figure}[htb]
$$\begin{array}{ccccc}
\PP_{n,i;j} &\stackrel{\bij}{\longleftrightarrow} & \M_{n,i;j} & \stackrel{\bijMG}{\longleftrightarrow} & \GG_{n,i;j}
\vspace{5pt} \\
\updownarrow\stackrel{\omega}{} && \updownarrow\stackrel{\omega}{} && \updownarrow\stackrel{\omega}{} \\
\O_{n,i;j} & \stackrel{\bijw}{\longleftrightarrow} & \Q_{n,i;j}  & \stackrel{\bijMGw}{\longleftrightarrow} & \H_{n,i;j}
\vspace{3pt}\\
 \begin{tikzpicture}[scale=0.45]
      \draw[fill,gray!10] (6,6)--(0,0)--(6,0)--(6,6);
      \draw[densely dotted] (-.5,0) -- (6,0); 
      \draw[densely dotted] (0,-.5) -- (0,6);
      \draw (6,6) -- (0,0) -- (6,0);
      \draw[fill,orange!20] (6,5)--(3,2)--(5,0)--(6,1);
      \draw[orange] (6,5)--(3,2)--(5,0)--(6,1);
      \filldraw[orange] (3,2) circle (2pt);
      \draw[orange] (3,2) node[above] {\footnotesize $(i,j)$};
      \draw[orange] (5,2) node {\footnotesize $\sh(i,j)$};
    \end{tikzpicture}     &&
 \begin{tikzpicture}[scale=0.45]
      \draw[fill,gray!10] (0,6)--(0,0)--(6,0)--(6,6)--(0,6);
      \draw[densely dotted] (-.5,0) -- (6,0); 
      \draw[densely dotted] (0,-.5) -- (0,6);
      \draw (0,6) -- (0,0) -- (6,0);
      \filldraw[orange] (3,2) circle (3pt);
      \draw[orange] (3,2) node[above] {\footnotesize $(i,j)$};
    \end{tikzpicture}    &&
\begin{tikzpicture}[scale=0.45]
      \draw[fill,gray!10] (-1,6)--(-1,0)--(4,0)--(4,6)--(-4,6);
      \draw[densely dotted] (-4,0) -- (4,0); 
      \draw[densely dotted] (0,-.5) -- (0,6);
      \draw (-4,0) -- (4,0);
      \draw[magenta] (-1,0) -- (-1,6);
      \draw[magenta] (-1,4) node[left] {\footnotesize $x=-\lfloor\frac{i}{2}\rfloor$};
      \filldraw[orange] (0,2) circle (3pt);
      \filldraw[orange] (1,2) circle (3pt);
      \draw[orange] (0,2) node[above] {\footnotesize $(0,j)$};
      \draw[orange] (1,2) node[above right] {\footnotesize $(1,j)$};
    \end{tikzpicture}
\end{array}$$
  \caption{
The bijections between pairs of lattice paths (top row) and their corresponding lattice walks (bottom row). 
The pictures show the shaded region where walks are constrained to, and the orange region where the walks can end.
In the picture on the right, the leftmost point of the walks must lie on the vertical magenta line.}
\label{fig:bijections_ij}
\end{figure}

Setting $j=0$ and taking the union over all $i$ with $0\le i\le n$ and $i\equiv n\pmod 2$, we obtain the diagram of bijections in Figure~\ref{fig:bijections}. Note that the top row consists of the bijections in Corollary~\ref{cor:PG}.

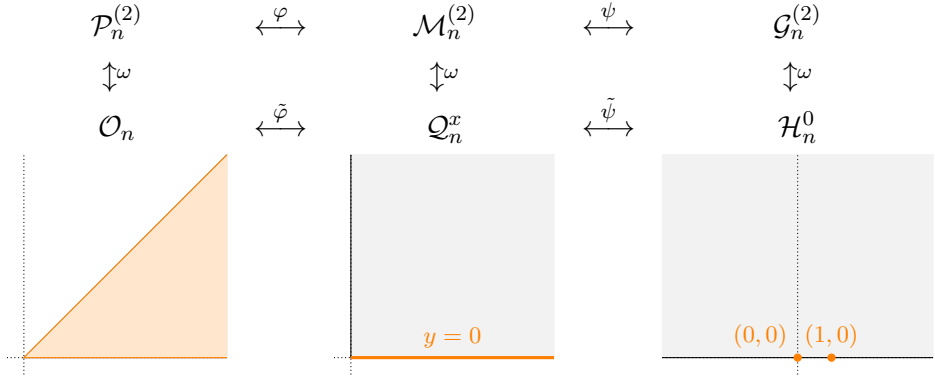
\begin{figure}[htb]
$$\begin{array}{ccccc}
\PP_n & \stackrel{\bij}{\longleftrightarrow} & \M_n & \stackrel{\bijMG}{\longleftrightarrow} & \GG_n 
\vspace{5pt} \\
\updownarrow\stackrel{\omega}{} && \updownarrow\stackrel{\omega}{} && \updownarrow\stackrel{\omega}{} \\
\O_n & \stackrel{\bijw}{\longleftrightarrow} & \Q^x_n & \stackrel{\bijMGw}{\longleftrightarrow} & \H^0_n
\vspace{3pt}\\
\begin{tikzpicture}[scale=0.45]
      \draw[fill,gray!10] (6,6)--(0,0)--(6,0)--(6,6);
      \draw[densely dotted] (-.5,0) -- (6,0); 
      \draw[densely dotted] (0,-.5) -- (0,6);
      \draw (6,6) -- (0,0) -- (6,0);
      \draw[fill,orange!20] (6,6)--(0,0)--(6,0)--(6,6);
      \draw[orange] (6,6) -- (0,0) -- (6,0);
\end{tikzpicture} & &
\begin{tikzpicture}[scale=0.45]
      \draw[fill,gray!10] (0,6)--(0,0)--(6,0)--(6,6)--(0,6);
      \draw[densely dotted] (-.5,0) -- (6,0); 
      \draw[densely dotted] (0,-.5) -- (0,6);
      \draw (0,6) -- (0,0) -- (6,0);
      \draw[orange, very thick] (0,0) -- (6,0);
      \draw[orange] (3,0) node[above] {\footnotesize $y=0$};
    \end{tikzpicture} & &
\begin{tikzpicture}[scale=0.45]
      \draw[fill,gray!10] (-4,6)--(-4,0)--(4,0)--(4,6)--(-4,6);
      \draw[densely dotted] (-4,0) -- (4,0); 
      \draw[densely dotted] (0,-.5) -- (0,6);
      \draw (-4,0) -- (4,0);
      \filldraw[orange] (0,0) circle (3pt);
      \filldraw[orange] (1,0) circle (3pt);
      \draw[orange] (0,0) node[above left] {\footnotesize $(0,0)$};
      \draw[orange] (1,0) node[above] {\footnotesize $(1,0)$};
    \end{tikzpicture}
\end{array}$$
  \caption{The corresponding bijections for the case of $j=0$ and arbitrary $i$.}
\label{fig:bijections}
\end{figure}

\subsection{The bijections in terms of walks}\label{sec:bijwalks}
Next we describe the bijections $\bijw$ and $\bijMGw$ in the bottom row of Figures~\ref{fig:bijections_ij} and~\ref{fig:bijections} directly in terms of walks.

We start with $\bijw:\Q_{n,i;j}\to\O_{n,i;j}$. Given $w\in\Q_{n,i;j}$, its image $w''=\bijw(w)$ is constructed as follows:

\begin{enumerate}
\item Let $w'$ be the path obtained by reflecting along a line of slope $1$ (that is, switching $N$ with $E$
and $S$ with $W$) all the steps of $w$ that end in the region $y>x$.

\item Let $w''$ be the path obtained from $w'$ by changing into $N$ steps all the $S$ steps of $w'$ that end at a lower $y$-coordinate than all the previous steps of $w'$.
\end{enumerate}

It is easy to check that this definition of $\bijw$ is equivalent to the definition of $\bij$ given in Section~\ref{sec:bij}. Figure~\ref{fig:bijw} gives an example of this construction, which is the translation to walks of the example given in Figure~\ref{fig:bij} for paths.

\begin{figure}[!ht]
  \begin{center}
  \begin{tabular}{l}
      \begin{tikzpicture}[scale=1]
    \def\E{(1,0)}     \def\W{(-1,0)}     \def\N{(0,1)}     \def\S{(0,-1)}
      \draw[densely dotted] (0,0)--(7,0);
      \draw[densely dotted] (0,0)--(0,4);
      \draw[dotted] (0,0)--(4,4);
      \draw[purple,thick] (0,0) coordinate(d0)
      -- ++\N coordinate(d1)
      -- ++\E coordinate(d2)
      -- ++(0,0.05)
      -- ++\W coordinate(d3)
      -- ++\N coordinate(d4)
      -- ++\E coordinate(d5)
      -- ++\E coordinate(d6)
      -- ++\E  coordinate(d7)
      -- ++\S coordinate(d8)
      -- ++\E coordinate(d9)
      -- ++\E  coordinate(d10)
      -- ++\E  coordinate(d11)
      -- ++(0,-0.05)
      -- ++\W coordinate(d12)
      -- ++\S++(0,0.05) coordinate(d13)
      -- ++\E coordinate(d14)
      -- ++(0,-0.05)
      -- ++\W  coordinate(d15)
      -- ++\W coordinate(d16)      
      -- ++\W++(0.05,0)  coordinate(d17)
      -- ++\N coordinate(d18)
      -- ++(-0.05,0)      
      -- ++\S  coordinate(d19)
      -- ++(-0.05,0)      
      -- ++\N coordinate(d20)      
      -- ++\E++(0.1,0) coordinate(d21)
      ;
       \foreach \x in {0,1,2,4,5,6,7,8,10,11,13,14,16,17} {
        \filldraw (d\x) circle (1.5pt);
      }
      \filldraw (d21)  ++(-3pt,-3pt) rectangle ++(5pt,5pt);
      \foreach \x in {0,...,20} {
       \pgfmathparse{\x+1}
       \xdef\y{\pgfmathresult}
        \draw[->,purple,thick] (d\x) -- (d\y);
      }
      \foreach \x in {0,3,7,19} {
       \pgfmathtruncatemacro\y{\x+1}
       \draw ($(d\x)!0.5!(d\y)$) node[left] {\footnotesize $\y$};
      }
      \foreach \x in {12,17} {
       \pgfmathtruncatemacro\y{\x+1}
       \draw ($(d\x)!0.5!(d\y)$) node[right] {\footnotesize $\y$};
      }
      \foreach \x in {1,11,14,15,16,18,20} {
       \pgfmathtruncatemacro\y{\x+1}
       \draw ($(d\x)!0.5!(d\y)$) node[below] {\footnotesize $\y$};
      }
      \foreach \x in {2,4,5,6,8,9,10,13} {
       \pgfmathtruncatemacro\y{\x+1}
       \draw ($(d\x)!0.5!(d\y)$) node[above] {\footnotesize $\y$};
      }
	  \draw[purple] (1,2.5) node {$w$};
      \draw[->,thick,dotted,yellow] (d0)--(d1);
      \draw[->,thick,dotted,yellow] (d2)+(0,0.05)--(d3);
      \draw[->,thick,dotted,yellow] (d3)--(d4);
      \draw[->,thick,dotted,yellow] (d4)--(d5);
      \draw (4,1) node[above] {\small $(i,j)$};
	  \draw (1,4) node {$\Q_{n,i;j}$};
      \draw (4,-1) node {$\updownarrow$};
\end{tikzpicture}\\
    \begin{tikzpicture}[scale=1]
    \def\E{(1,0)}     \def\W{(-1,0)}     \def\N{(0,1)}     \def\S{(0,-1)}
      \draw[densely dotted] (0,0)--(9,0);
      \draw[densely dotted] (0,0)--(0,2);
      \draw[dotted] (0,0)--(2,2);
      \draw[purple,thick] (0,0) coordinate(d0)
      -- ++\E coordinate(d1)
      -- ++\E coordinate(d2)
      -- ++\S coordinate(d3)
      -- ++\E coordinate(d4)
      -- ++\N coordinate(d5)
      -- ++\E coordinate(d6)
      -- ++\E  coordinate(d7)
      -- ++\S coordinate(d8)
      -- ++\E coordinate(d9)
      -- ++\E  coordinate(d10)
      -- ++\E  coordinate(d11)
      -- ++(0,-0.05)
      -- ++\W coordinate(d12)
      -- ++\S++(0,0.05) coordinate(d13)
      -- ++\E coordinate(d14)
      -- ++(0,-0.05)
      -- ++\W  coordinate(d15)
      -- ++\W coordinate(d16)      
      -- ++\W++(0.05,0)  coordinate(d17)
      -- ++\N coordinate(d18)
      -- ++(-0.05,0)      
      -- ++\S  coordinate(d19)
      -- ++(-0.05,0)      
      -- ++\N coordinate(d20)      
      -- ++\E++(0.1,0) coordinate(d21)
      ;
       \foreach \x in {0,1,2,3,4,5,6,7,8,10,11,13,14,16,17} {
        \filldraw (d\x) circle (1.5pt);
      }
      \filldraw (d21)  ++(-3pt,-3pt) rectangle ++(5pt,5pt);
      \foreach \x in {0,...,20} {
       \pgfmathparse{\x+1}
       \xdef\y{\pgfmathresult}
        \draw[->,purple,thick] (d\x) -- (d\y);
      }
      \foreach \x in {2,4,7,19} {
       \pgfmathtruncatemacro\y{\x+1}
       \draw ($(d\x)!0.5!(d\y)$) node[left] {\footnotesize $\y$};
      }
      \draw (0,1) node[left] {$\ $};
      \foreach \x in {12,17} {
       \pgfmathtruncatemacro\y{\x+1}
       \draw ($(d\x)!0.5!(d\y)$) node[right] {\footnotesize $\y$};
      }
      \foreach \x in {11,14,15,16,18,20} {
       \pgfmathtruncatemacro\y{\x+1}
       \draw ($(d\x)!0.5!(d\y)$) node[below] {\footnotesize $\y$};
      }
      \foreach \x in {0,1,3,5,6,8,9,10,13} {
       \pgfmathtruncatemacro\y{\x+1}
       \draw ($(d\x)!0.5!(d\y)$) node[above] {\footnotesize $\y$};
      }
	  \draw[purple] (1,0.5) node {$w'$};
      \draw[->,ultra thick,green] (d2)--(d3);
      \draw[->,ultra thick,green] (d12)--(d13);
      \draw[->,thick,dotted,yellow] (d0)--(d1);
      \draw[->,thick,dotted,yellow] (d2)--(d3);
      \draw[->,thick,dotted,yellow] (d3)--(d4);
      \draw[->,thick,dotted,yellow] (d4)--(d5);
      \filldraw (4,1) circle (1pt);
	  \draw (4,1) node[above] {\small $(i,j)$};
      \draw (4,-2.5) node {$\updownarrow$};
\end{tikzpicture}\\
    \begin{tikzpicture}[scale=1]
    \def\E{(1,0)}     \def\W{(-1,0)}     \def\N{(0,1)}     \def\S{(0,-1)}
      \draw[fill,orange!20] (7,4)--(4,1)--(5,0)--(8,3);
      \draw[densely dotted] (0,0)--(9,0);
      \draw[densely dotted] (0,0)--(0,4.5);
      \draw[dotted] (0,0)--(4.5,4.5);
      \draw[orange] (7,4)--(4,1)--(5,0)--(8,3);
	  \draw (7,3) node {\small $\sh(i,j)$};      
      
      \draw[purple,thick] (0,0) coordinate(d0)
      -- ++\E coordinate(d1)
      -- ++\E coordinate(d2)
      -- ++\N coordinate(d3)
      -- ++\E coordinate(d4)
      -- ++\N coordinate(d5)
      -- ++\E coordinate(d6)
      -- ++\E  coordinate(d7)
      -- ++\S coordinate(d8)
      -- ++\E coordinate(d9)
      -- ++\E  coordinate(d10)
      -- ++\E  coordinate(d11)
      -- ++(0,+0.05)
      -- ++\W coordinate(d12)
      -- ++\N++(0,-0.05) coordinate(d13)
      -- ++\E coordinate(d14)
      -- ++(0,0.05)
      -- ++\W  coordinate(d15)
      -- ++\W coordinate(d16)      
      -- ++\W++(0.05,0)  coordinate(d17)
      -- ++\N coordinate(d18)
      -- ++(-0.05,0)      
      -- ++\S  coordinate(d19)
      -- ++(-0.05,0)      
      -- ++\N coordinate(d20)      
      -- ++\E++(0.1,0) coordinate(d21)
      ;
       \foreach \x in {0,1,2,3,4,5,6,7,8,10,11,13,14,16,17} {
        \filldraw (d\x) circle (1.5pt);
      }
      \filldraw (d21)  ++(-3pt,-3pt) rectangle ++(5pt,5pt);
      \foreach \x in {0,...,20} {
       \pgfmathparse{\x+1}
       \xdef\y{\pgfmathresult}
        \draw[->,purple,thick] (d\x) -- (d\y);
      }
      \foreach \x in {2,4,7,19} {
       \pgfmathtruncatemacro\y{\x+1}
       \draw ($(d\x)!0.5!(d\y)$) node[left] {\footnotesize $\y$};
      }
      \draw (0,1) node[left] {$\ $};
      \foreach \x in {12,17} {
       \pgfmathtruncatemacro\y{\x+1}
       \draw ($(d\x)!0.5!(d\y)$) node[right] {\footnotesize $\y$};
      }
      \foreach \x in {3,5,6,8,9,10,13,18} {
       \pgfmathtruncatemacro\y{\x+1}
       \draw ($(d\x)!0.5!(d\y)$) node[below] {\footnotesize $\y$};
      }
      \foreach \x in {0,1,11,14,15,16,20} {
       \pgfmathtruncatemacro\y{\x+1}
       \draw ($(d\x)!0.5!(d\y)$) node[above] {\footnotesize $\y$};
      }
	  \draw[purple] (3.5,2.5) node {$w''$};
      \draw[->,ultra thick,green] (d2)--(d3);
      \draw[->,ultra thick,green] (d12)--(d13);
      \filldraw (4,1) circle (1pt);
	  \draw (4,1) node[above] {\small $(i,j)$};
	  \draw (1,4) node {$\O_{n,i;j}$};
\end{tikzpicture}
  \end{tabular}
  \end{center}
  \caption{The bijection $\bijw:\Q_{n,i;j}\to\O_{n,i;j}$. The steps that are changed in going between $w$ and $w'$ are dotted in yellow. The steps that changed in going between $w'$ and $w''$ are thicker and green.}\label{fig:bijw}
\end{figure}
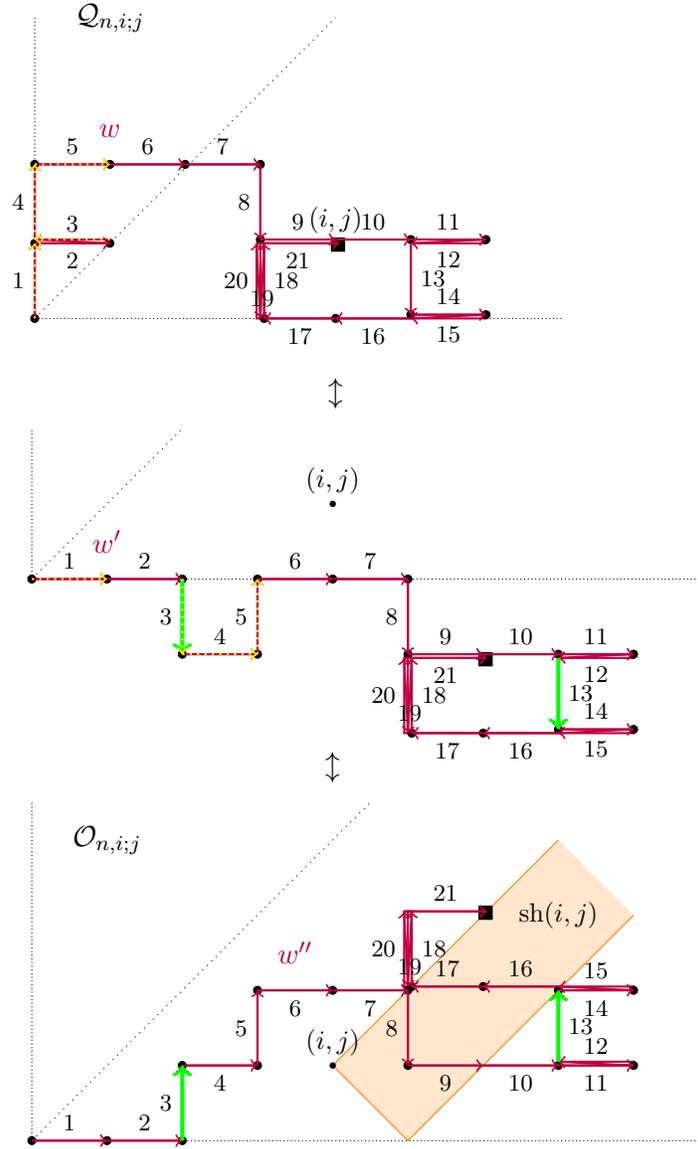

\medskip

Next we describe the bijection $\bijMGw:\Q_{n,i;j}\to\H_{n,i;j}$. A walk can be decomposed uniquely as a sequence of (vertical) $N$ and $S$ steps interleaved with a sequence of (horizontal) $E$ and $W$ steps. 

A lattice walk belongs to $\Q^x_n$ if and only if its $NS$-subsequence, after substituting $U$ for $N$ and $D$ for $S$, is a Dyck path, and its $EW$-subsequence, after substituting $U$ for $E$ and $D$ for $W$, is a Dyck path prefix.
Similarly, a walk belongs to $\H^0_n$ if and only if its $NS$-subsequence produces a Dyck path, and its $EW$-subsequence produces a Grand Dyck path. 

The bijection $\bijMGw:\Q^x_{n}\to\H^0_{n}$ has a simple description in terms of this decomposition. Given $w\in\Q^x_n$, consider the Dyck path prefix determined by its $EW$-subsequence, and apply $\xi$ to it to obtain a Grand Dyck path  of the same length (with steps $E,W$ playing the role of $U,D$). The resulting walk belongs to $\H^0_n$ because it consists of a Dyck path interleaved with a Grand Dyck path.
More generally, a similar description can be given for $\bijMGw:\Q_{n,i;j}\to\H_{n,i;j}$. The only difference is that the $NS$-subsequence gives a Dyck path prefix ending at height $j$, and the $EW$-subsequence is mapped via $\xi$ from a Dyck path prefix ending at height $i$ to a Grand Dyck path with lowest point at height $-\lfloor i/2 \rfloor$.

\section{Connections to work in the literature}\label{sec:related}
In this section we discuss related results in the literature, and we show how our work applies in different contexts.

\subsection{Plane partitions}
A non-bijective proof of the fact that $|\PP_n|=|\GG_n|$, which we proved bijectively in Corollary~\ref{cor:PG},
and more generally of Equation~\ref{eq:PG}, follows from a result of Proctor~\cite{Proctor} on plane partitions. 
Recall that a plane partition is a two-dimensional array of nonnegative integers $a_{i,j}$ weakly decreasing in rows and columns.
In the rest of this section we assume that $p\ge q$.
A plane partition is said to be contained in the rectangle shape $(p^q)$ if the range for the indices is $1\le i\le q$ and $1\le j\le p$, and contained in the shifted shape $[p + q - 1,p + q - 3,\dots, p - q + 1]$ if the range is
$1\le i\le q$ and $i\le j\le p+q-i$. A plane partition has part size bounded by $k$ if $0\le a_{i,j}\le k$ for all $i,j$.

It is easy to see that tuples of paths in $\A^{(k)}_{p+q}$ ending at $(p+q,p-q)$ are in bijection with plane partitions contained in the rectangle shape $(p^q)$ with part size bounded by~$k$. The idea is that for each $1\le i\le k$, the boundary between entries smaller than $i$ and entries larger than or equal to $i$ in the array determines a path in the $k$-tuple.
Similarly, tuples
of paths in $\P^{(k)}_{p+q}$ ending at height at least $p-q$ are in bijection with
(shifted) plane partitions contained in the shifted shape $[p + q - 1,p + q - 3,\dots, p - q + 1]$ with part size bounded by $k$.
The main result in~\cite{Proctor} is that these two sets of plane partitions have the same cardinality.
The proof uses combinatorial descriptions of finite-dimensional representations of semisimple Lie algebras, and it is not bijective.

Using the above correspondences between tuples of paths and plane partitions, Corollary~\ref{cor:PG} gives a bijective proof of Proctor's result for $k=2$ and $p=q$. In fact, we can use a slight modification of $\bijMGw$ to get rid of the restriction $p=q$ and provide a bijective proof of Proctor's result for $k=2$ in a somewhat more general form. In terms of pairs of nested paths, letting $n=p+q$ and $s=p-q$, Proctor's result for $k=2$ states that, for any $s\equiv n\pmod n$, the number of pairs
$(P,Q)\in\PP_n$ with $h(Q)\ge s$ equals the number of pairs $(P,Q)\in\AA_n$ with $h(P)=h(Q)=s$. When translated in terms of walks using $\omega$, it states that the number of walks in $\O_n$ ending in the region $y\le x-s$ equals
the number of walks in $\H_n$ ending at $(s,0)$. 

First we modify the bijection $\xi:\P_n\to\G_n$ by introducing a parameter $s\ge0$ as follows. 
Given a path $P\in\P_n$ with $h(P)=i$, where $i\ge s$ and $i\equiv s\pmod 2$, let $\xi_s(P)$ 
be the path obtained by changing the leftmost $(i-s)/2$ unmatched $U$ steps of $P$ into $D$ steps.
It is clear that $\xi_s(P)$ ends at height $s$. The same argument that proves that $\xi$ is a bijection
between $\P_n$ and $\G_n$ shows that, for any $i$ as above, the map $\xi_s$ is a bijection between paths in $\P_n$ ending at height $i$ and paths in $\A_n$ ending at height $s$ whose lowest point is at height $-(i-s)/2$. 

By modifying $\bijMGw$ accordingly, we define a map $\bijMGw_s$ as follows. Let $j$ and $n$ be such that $s+j\equiv n\pmod 2$.
Given a walk $w\in\Q_{n,i;j}$, where $i\ge s$ and $i\equiv s\pmod 2$, apply $\xi_s$ to its $EW$-subsequence (with steps $E,W$ playing the role of $U,D$). The resulting walk $\bijMGw_s(w)$ ends at $(s,j)$, and its leftmost point has $x$-coordinate  $-(i-s)/2$. In fact, $\bijMGw_s$ is a bijection between $\Q_{n,i;j}$ and the set of walks in $\H_n$ ending at $(s,j)$ whose leftmost point lies on  $x=-(i-s)/2$.
Taking the union over all $i$ with $i\ge s$ and $i\equiv s\pmod 2$, the map $\bijMGw_s$ gives a bijection between $\bigcup_{i\ge s}\Q_{n,i;j}$ and the set of walks in $\H_n$ ending at $(s,j)$.
Thus, the composition $\bijMGw_s\circ\bijw^{-1}$ is a bijection between the disjoint union $\bigsqcup_{i\ge s}\O_{n,i;j}$ (note that in general the sets $\O_{n,i;j}$ are not disjoint) and the set of walks in $\H_n$ ending at $(s,j)$.

Restricted to the case $j=0$, our modified map $\bijMGw_s\circ\bijw^{-1}$ is a bijection between walks in $\O_n$ ending in $y\le x-s$ and walks in $\H_n$ ending at $(s,0)$, as illustrated in Figure~\ref{fig:proctor}, proving Proctor's result for $k=2$.

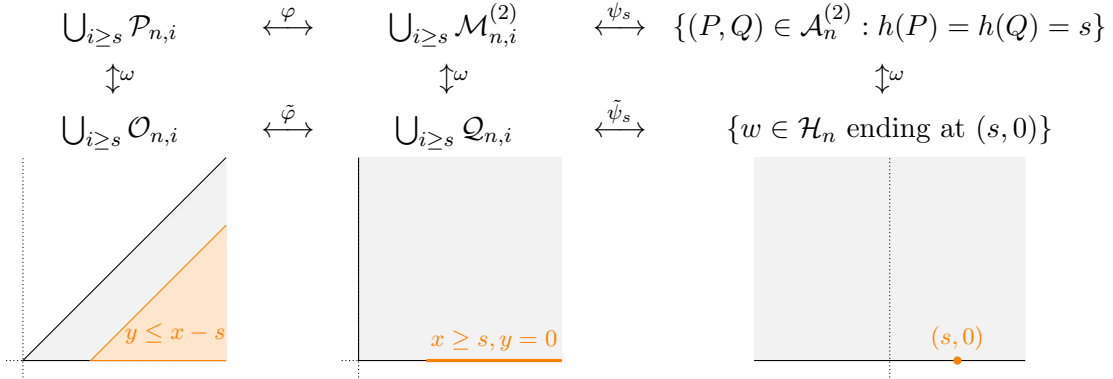
\begin{figure}[htb]
$$\begin{array}{ccccc}
\bigcup_{i\ge s}\P_{n,i} & \stackrel{\bij}{\longleftrightarrow} & 
\bigcup_{i\ge s}\M_{n,i} & \stackrel{\bijMG_s}{\longleftrightarrow} & \{(P,Q)\in\AA_n: h(P)=h(Q)=s\}
\vspace{5pt} \\
\updownarrow\stackrel{\omega}{} && \updownarrow\stackrel{\omega}{} && \updownarrow\stackrel{\omega}{} \\
\bigcup_{i\ge s}\O_{n,i} & \stackrel{\bijw}{\longleftrightarrow} & \bigcup_{i\ge s}\Q_{n,i} & \stackrel{\bijMGw_s}{\longleftrightarrow} & \{w\in\H_n\text{ ending at }(s,0)\}
\vspace{3pt}\\
\begin{tikzpicture}[scale=0.45]
      \draw[fill,gray!10] (6,6)--(0,0)--(6,0)--(6,6);
      \draw[densely dotted] (-.5,0) -- (6,0); 
      \draw[densely dotted] (0,-.5) -- (0,6);
      \draw (6,6) -- (0,0) -- (6,0);
      \draw[fill,orange!20] (6,4)--(2,0)--(6,0)--(6,4);
      \draw[orange] (6,4) -- (2,0) -- (6,0);
      \draw[orange] (4.5,.8) node {\footnotesize $y\le x-s$};
\end{tikzpicture} & &
\begin{tikzpicture}[scale=0.45]
      \draw[fill,gray!10] (0,6)--(0,0)--(6,0)--(6,6)--(0,6);
      \draw[densely dotted] (-.5,0) -- (6,0); 
      \draw[densely dotted] (0,-.5) -- (0,6);
      \draw (0,6) -- (0,0) -- (6,0);
      \draw[orange, very thick] (2,0) -- (6,0);
      \draw[orange] (4,0) node[above] {\footnotesize $x\ge s,y=0$};
    \end{tikzpicture} & &
\begin{tikzpicture}[scale=0.45]
      \draw[fill,gray!10] (-4,6)--(-4,0)--(4,0)--(4,6)--(-4,6);
      \draw[densely dotted] (-4,0) -- (4,0); 
      \draw[densely dotted] (0,-.5) -- (0,6);
      \draw (-4,0) -- (4,0);
      \filldraw[orange] (2,0) circle (3pt);
      \draw[orange] (2,0) node[above] {\footnotesize $(s,0)$};
    \end{tikzpicture}
\end{array}$$
  \caption{The bijections proving Proctor's result for $k=2$.}
\label{fig:proctor}
\end{figure}

\subsection{Partially ordered sets}

In Stanley's book \cite[Exercise 3.47(f)]{EC1} we find another appearence of Proctor's result in a slightly different form. This exercise asks to show that the order polynomials of two posets are the same. The first poset is
the product of two chains of length $q$ and $p$, and the second is the poset of pairs
$\{(i, j): 1 \le i \le j \le p+q - i, 1 \le i \le q\}$
ordered by $(i,j)\le (i',j')$ if $i\le i'$ and $j\le j'$.
Recall that the order polynomial of a poset evaluated at $k$ is the number of order-preserving maps from the poset to a $k$-element chain. Such maps correspond to the plane partitions considered by Proctor.
In~\cite{Proctor}, Proctor also shows that both posets have the same number of $l$-element chains for all $l$ (equivalently, they have the same zeta-polynomials). A different proof of this fact using symmetric functions is also given by Stembridge~\cite{Stembridge}.

Very closely related to Proctor's result for plane partitions is the following result of Haiman. In~\cite[Prop. 8.11]{Haiman}, Haiman gives a bijection between {\em standard} Young tableaux of rectangular shape $(p^q)$ and {\em standard} Young tableaux of shifted shape $[p + q - 1,p + q - 3,\dots, p - q + 1]$. His bijection is based on Sh\"utzenberger's {\it jeu de taquin}~\cite{Sch}. The fact that these shapes have the same number of standard Young tableaux follows from Proctor's result. However, it is not clear whether it is possible to generalize Haiman's bijection using {\it jeu de taquin} to plane partitions.

\subsection{Watermelons and stars with a wall}

Tuples of paths in $\A^{(k)}_{p+q}$ ending at $(p+q,p-q)$ are called \emph{watermelons} in~\cite{GOV}, where they are enumerated  using a correspondence between watermelons and certain semistandard Young tableaux. A determinantal formula
can also be obtained using the Gessel--Viennot method~\cite{GV}. Alternatively, since these tuples are in bijection with plane partitions contained in the rectangle shape $(p^q)$ with part size bounded by $k$ (or, as they are commonly described,
plane partitions that fit in a $p\times q\times k$ box), their number is given by the following formula of MacMahon~\cite{MacMahon} (see also \cite{Macdonald}):
\begin{equation}\label{eq:MacMahon}
\prod_{i=1}^p\prod_{j=1}^q\prod_{l=1}^k \frac{i+j+l-1}{i+j+l-2},
\end{equation}
which we used to obtain the right-hand side of~\eqref{eq:Gk}.

Tuples of nested lattice paths of the same length (elements of $\A^{(k)}_n$ in our terminology) are called \emph{stars} in~\cite{GOV,KGV}. Krattenthaler, Guttmann and Viennot~\cite{KGV} consider stars with a {\em wall restriction}, and they show~\cite[Theorem 7]{KGV} that the number of tuples in $\P^{(k)}_{p+q}$ ending at height at least $p-q$ is
also given by Equation~\eqref{eq:MacMahon}.
Their proof, which is based on Proctor's proof, uses a correspondence between these tuples of paths and symplectic tableaux 
(that is, semistandard Young tableau where entries in row $r$ are at least $2r-1$ for all $r$)
with entries bounded by $p+q-1$ having at most $p$ rows and at most $k$ columns. The proof then follows from an identity
relating symplectic characters and Schur functions of rectangular shape, which is a special case of an identity for universal characters.

\subsection{Walks in the octant}
Walks in the first octant with specific endpoints have been enumerated by Bousquet-M\'elou and Mishna~\cite[Section 5.3]{BMM} using functional equations and the kernel method. The walks studied in~\cite{BMM} are walks in the first quadrant with slightly different types of steps, but they are trivially equivalent to ours through a linear transformation.

Bousquet-M\'elou and Mishna give a formula for the number of walks in $\O_n$ with a given endpoint, and from it, using Gosper's algorithm to sum hypergeometric sequences, they deduce the following closed expressions, which we write in terms of the Catalan numbers $C_m$:
\begin{align}|\{w\in \O_n: w\text{ ends on the $x$-axis}\}|&=\begin{cases} C_m C_{m+1} & \text{if }n=2m, \\
C_{m+1}^2 & \text{if }n=2m+1, \end{cases} \label{eq:Oxaxis}\\ 
|\{w\in \O_{2m}: w\text{ ends on $y=x$}\}|&=|\O_{2m,0}|=C_m C_{m+1}, \label{eq:Odiag}\\
|\O_n|&=\begin{cases} (2m+1)C_m^2 & \text{if }n=2m, \\
(2m+1)C_mC_{m+1} & \text{if }n=2m+1. \end{cases} \label{eq:O}
\end{align}

Of course, since $|\O_n|=|\PP_n|=|\GG_n|$, we know that the formula~\eqref{eq:O} agrees with the two expressions in Equation~\eqref{eq:Gk} for $k=2$, and also with the formula
$$|\GG_n|=|\H^0_n|=\sum_{l=0}^{\fn2} \binom{n}{2l} |\D_{2l}| |\G_{n-2l}| 
=\sum_{l=0}^{\fn2} \frac{1}{l+1}\binom{n}{l,l,\fn2-l,\cn2-l},
$$
which follows from the decomposition of walks in $\H^0_n$ obtained by separating their $NS$- and $EW$-subsequences, as described in the Section~\ref{sec:bijwalks}.

The nice product of Catalan numbers in Equation~\eqref{eq:Oxaxis} was first discovered by Gouyou-Beau\-champs~\cite{G-B86}, who gave a combinatorial proof of this formula. Interpreting walks ending on the $x$-axis, via $\omega$, as pairs of Dyck path prefixes having the same endpoint, he first expresses their number as a sum over all possible endpoints.
For each given endpoint, non-crossing pairs of paths are easily counted using the Gessel--Viennot method~\cite{GV}, by subtracting crossing pairs from all pairs, for which there are simple formulas. The sums over all possible endpoints are then interpreted as concatenations of paths, and finally a clever involution is applied to cancel positive and negative terms and obtain the formula~\eqref{eq:Oxaxis}. 
Additionally, Gouyou-Beauchamps~\cite{G-B89} provides a bijection between walks in $\O_n$ ending on the $x$-axis and standard Young tableaux with $n$ cells having at most 4 rows.

On the other hand, it seems that no combinatorial proof of Equation~\eqref{eq:Odiag} is known. This equation counts walks ending on the diagonal, or equivalently, via $\omega$,
pairs of nested paths $(P,Q)$ where $P\in\P_{2m}$ and $Q\in\D_{2m}$. In \cite[Section 7.1]{BMM}, Bousquet-M\'elou and Mishna leave open the question of
finding a bijective proof of the fact that expressions~\eqref{eq:Oxaxis} and~\eqref{eq:Odiag} agree for even $n$, that is, that the number of walks in $\O_{n}$ ending on the $x$-axis equals the number of those ending on the diagonal. Next we provide a bijective proof of Equation~\eqref{eq:Odiag}.

\begin{corollary}\label{cor:diagonal}
There is an explicit bijection between the set $\O_{2m,0}$ of walks in the first octant ending on the diagonal and the set $\D_{2m}\times\D_{2(m+1)}$ of pairs of Dyck paths.
\end{corollary}

\begin{proof}
Our bijection $\bijw:\O_{2m}\to\Q^x_{2m}$ restricts to a bijection between $\O_{2m,0}$ and $\Q^x_{2m,0}$, the set of walks in the first quadrant of length $2m$ that end at the origin. Composing $\bijw$ with any of the known bijections between $\Q^x_{2m,0}$ and pairs of Dyck paths provides the desired bijection.
The first bijection between $\Q^x_{2m,0}$ and $\D_{2m}\times\D_{2(m+1)}$ was constructed recursively by Cori, Dulucq and Viennot~\cite{CDV}, and later a more direct bijection passing through certain planar maps was given by Bernardi~\cite{Ber}. 

We point out that Guy, Krattenthaler and Sagan~\cite{GKS} gave another simple proof of the fact that $|\Q^x_{2m,0}|=C_m C_{m+1}$ using the reflection principle, and thus involving negative signs.
\end{proof}

\section{Proof of Theorem~\ref{thm:bijMP}}\label{sec:proofs}

The proof of Theorem~\ref{thm:bijMP} will follow from Lemmas~\ref{lem:image}, \ref{lem:injective} and~\ref{lem:surjective} below. It will be convenient to introduce some notation for the proofs, and to refer to the example in Figure~\ref{fig:bij}. 
Let $\R$ be the set of $U$ steps of $Q$ that end on the $x$-axis, which we call {\em lower returns} of $Q$, and let $r=|\R|$. 
Let $\abs[Q]$ be the path obtained by flipping the steps of $Q$ below the $x$-axis. Note that $h_a(\abs[Q])=|h_a(Q)|$ for all $a$. 
For any nonnegative integer $a$, let $\R^{\le a}$ be the set of steps in $\R$ to the left of $x=a$. 
Define $\chi^{\le a}$ similarly.

\begin{lemma}\label{lem:QQ'}
The transformation $Q\mapsto Q'$ in step~1 of the description of $\bij$ is a bijection between paths $Q\in\A_n$ with $h(Q)\ge0$ having $r$ lower returns, and paths $Q'\in\P_n$ with $h(Q')\ge 2r$.
Additionally, \begin{equation}\label{eq:QQ'}h_a(Q')=h_a(\abs[Q])+2|\R^{\le a}|\end{equation} for all $a$.
\end{lemma}

\begin{proof}
The path $Q'$ is obtained by flipping all the steps of $Q$ below the $x$-axis except for those in $\R$.
Thus, $Q'$ and $\abs[Q]$ differ precisely in the steps in $\R$, which are $U$ steps in $Q'$.
Equation~\ref{eq:QQ'} follows, and in particular $h(Q')=h(\abs[Q])+2r=h(Q)+2r$.

Now we show that this map is a bijection. If the steps in $\R$ are $u_0,u_1,\dots,u_{r-1}$ from left to right, then step $u_l$ becomes the rightmost $U$ step of $Q'$ rising from height $2l+1$ to height $2l+2$.
The transformation $Q\mapsto Q'$ flips, for each $0\le l< r$, the fragment of $Q$ between $u_l$ and the previous point at height $0$ (not including step $u_l$ itself), and this fragment becomes the piece of $Q'$ between the last point at height $2l$ and the last point at height $2l+1$. Thus, one recovers $Q$ from $Q'$ by flipping these fragments again. 
\end{proof}

\begin{lemma}\label{lem:image}
$\bij(\M_{n,i;j})\subseteq \PP_{n,i;j}$.
\end{lemma}

\begin{proof}
Let $(P,Q)\in \M_{n,i;j}$ and let $(\tP,\tQ)=\bij(P,Q)$. We will prove that $(\tP,\tQ)\in\PP_{n,i;j}$ by showing that 
$\tP\ge\tQ\ge0$ and that $h(\tP)$ and $h(\tQ)$ satisfy the required inequalities.

The fact that $\tP\ge\tQ$ is clear from step 2 in the description of $\bij$. Indeed, 
since the path $(\tP-\tQ)/2$ is obtained by turning all the unmatched $D$s of $(P-Q')/2$ into $U$s,
we have that $(\tP-\tQ)/2\ge0$, and so $\tP\ge\tQ$.

Next we show that $\tQ\ge0$. It follows from Equation~\eqref{eq:QQ'} that $h(Q')=h(Q)+2r=i-j+2r$ and that $Q'\ge0$.
Since $-P\le Q\le P$, or equivalently $\abs[Q]\le P$, it also follows that $h_a(Q')-h_a(P)\le 2|\R^{\le a}|$. The right hand side of this inequality is weakly increasing in $a$, and so
\begin{equation}\label{eq:chiR}
2|\chi^{\le a}|=\max_{0\le b\le a}\{h_b(Q')-h_b(P)\} \le 2|\R^{\le a}|,
\end{equation}
where the left equality is a consequence of the definition of $\chi$.

Using the definition of $\tQ$ and Equations~\eqref{eq:chiR} and~\eqref{eq:QQ'}, in this order, we obtain
$$h_a(\tQ)=h_a(Q')-2|\chi^{\le a}|\ge h_a(Q')-2|\R^{\le a}|\ge0$$
for every $a$, and so $\tQ\ge0$.

\smallskip

Next we show that the ending heights of $\tP$ and $\tQ$ are in the required intervals.
Equation~\eqref{eq:chiR} for $a=n$ implies that $|\chi|\le r$. Noting that the ending height of the path $(Q'-P)/2$ is 
$$\frac{h(Q')-h(P)}{2}=\frac{i-j+2r-(i+j)}{2}=r-j$$
and that its maximum height is $|\chi|$, we obtain
\begin{equation}\label{eq:chi}
r-j\le |\chi| \le r.
\end{equation}

By construction of $\tP$ and $\tQ$, we have that $$h(\tP)=h(P)+2|\chi|=i+j+2|\chi|\ge i+j$$ and that $$h(\tQ)=h(Q')-2|\chi|=i-j+2(r-|\chi|),$$
and so $i-j\le h(\tQ) \le i+j$ by Equation~\eqref{eq:chi}.
\end{proof}

\begin{lemma}\label{lem:injective}
The map $\bij:\M_{n,i;j}\to\PP_{n,i;j}$ is injective.
\end{lemma}

\begin{proof}
We show that the transformations $Q\mapsto Q'$ and $(P,Q')\mapsto(\tP,\tQ)$ in the definition of $\bij$ are invertible. Note that the values $i$ and $j$ are fixed, so we can use the knowledge of $i+j$ and $i-j$ when inverting these transformations.

Given $(\tP,\tQ)\in\PP_{n,i;j}$, in order to recover $(P,Q')$ it is enough to determine the set $\chi$ of positions of the steps that have been flipped.
Recall that the paths $\tP$ and $\tQ$ differ from $P$ and $Q'$, respectively, exactly in the positions of the unmatched $D$ steps of $(P-Q')/2$. In these positions, $P$ and $\tQ$ have $D$ steps, whereas $\tP$ and $Q'$ have $U$ steps.
Flipping the unmatched $D$ steps of $(P-Q')/2$ turns these steps into the leftmost $|\chi|$ unmatched $U$ steps of $(\tP-\tQ)/2$.

Note that $|\chi|$ can be easily determined from $\tP$, since $h(\tP)=i+j+2|\chi|$, so $|\chi|=(h(\tP)-i-j)/2$. Thus, the 
leftmost $(h(\tP)-i-j)/2$ unmatched $U$s of $(\tP-\tQ)/2$ determine $\chi$, and flipping the corresponding steps in $\tP$ and $\tQ$ we recover $P$ and $Q'$, respectively.

The fact that the transformation $Q\mapsto Q'$ is invertible follows from Lemma~\ref{lem:QQ'}, noticing that the value $r=|\R|$ can be obtained from $Q'$ using that $h(Q')=i-j+2r$, since $i-j$ is known.
\end{proof}

The proof of Lemma~\ref{lem:injective} yields a description of the inverse map $\bij^{-1}$. Given $(\tP,\tQ)\in\bij(\M_{n,i;j})\subseteq \PP_{n,i;j}$, its preimage $(P,Q)=\bij^{-1}(\tP,\tQ)$ can be obtained as follows:

\begin{enumerate}
\item Let $\chi$ be the set of positions of the leftmost $(h(\tP)-i-j)/2$ unmatched $U$ steps of $(\tP-\tQ)/2$.
Let $P$ and $Q'$ be the paths obtained by flipping the steps in $\chi$ of $\tP$ and $\tQ$, respectively.

\item Let $r=(h(Q')-i+j)/2$. For $0\le l< r$, let $Q'_l$ the fragment of $Q'$ between the rightmost point at height $2l$ and the rightmost point at height $2l+1$. Let $Q$ be the path obtained from $Q'$ by flipping the steps in each $Q'_l$.
\end{enumerate}

\begin{lemma}\label{lem:surjective}
The map $\bij:\M_{n,i;j}\to\PP_{n,i;j}$ is surjective.
\end{lemma}

\begin{proof}
We will show that when the above construction for $\bij^{-1}$ is applied to an arbitrary pair $(\tP,\tQ)\in\PP_{n,i;j}$, it produces a pair $(P,Q)\in\M_{n,i;j}$ such that $\bij(P,Q)=(\tP,\tQ)$.

From step~1 in the description of $\bij^{-1}$, we see that $$h(P)=h(\tP)-2|\chi|=i+j,$$ and that $h(Q')=h(\tQ)+2|\chi|\ge h(\tQ)\ge i-j$, which implies that the value of $r$ defined in step~2 is nonnegative. Additionally, since $i\ge j$, we have that $2r\le h(Q')$, and thus the pieces $Q'_l$ for $0\le l<r$ are non-empty. After they are flipped to build $Q$, we get $$h(Q)=h(Q')-2r=i-j$$ as desired.

It remains to show that the paths $P$ and $Q$ produced by $\bij^{-1}$ satisfy $-P\le Q\le P$.
By construction of $Q'$, we have $h_a(Q')=h_a(\tQ)+2|\chi^{\le a}|\ge 2|\chi^{\le a}|$ for all $a$. Since $|\chi^{\le a}|$ is increasing in $a$, it follows that
\begin{equation}\label{eq:min}\min_{a\le b\le n} h_b(Q')\ge 2|\chi^{\le a}|.\end{equation}
By Lemma~\ref{lem:QQ'}, the transformation $Q\mapsto Q'$ is a bijection whose inverse is given by step~2 of the description of $\bij^{-1}$.
Let $\R$ be the set whose elements are, for each $0\le l<r$, the rightmost $U$ step of $Q'$ rising from height $2l+1$ to height $2l+2$. The proof of Lemma~\ref{lem:QQ'} shows that the steps $\R$ become the lower returns of $Q$. Defining $\R^{\le a}$ accordingly, it is clear from the definition that
\begin{equation}\label{eq:min2}2|\R^{\le a}|\ge \min_{a\le b\le n} h_b(Q').\end{equation}

By construction of $P$ and $Q'$ in step~1 of the description of $\bij^{-1}$, we have that $h_a(Q')-h_a(P)\le 2|\chi^{\le a}|$, which,
in combination with inequalities~\eqref{eq:min} and~\eqref{eq:min2}, implies that $h_a(Q')-h_a(P)\le 2|\R^{\le a}|$. 
Now we use equation~\eqref{eq:QQ'} to conclude that
$$h_a(\abs[Q])=h_a(Q')-2|\R^{\le a}|\le h_a(P)$$
for all $a$. This proves that $\abs[Q]\le P$, or equivalently, $-P\le Q\le P$.
\end{proof}

\section{Open problems}\label{sec:open}
The main contribution of this article is a bijection between $\PP_n$ and $\GG_n$, which extends the known bijections between $\P_n$ and $\G_n$.
Generalizing our bijection to $k$-tuples of paths remains an open problem.

\begin{problem}
Find an explicit bijection between $\P^{(k)}_n$ and $\G^{(k)}_n$ for $k\ge3$. 
\end{problem}

Finding an extension of Haiman's bijection mentioned in Section~\ref{sec:related} to plane partitions would give a bijection between $\P^{(k)}_n$ and $\G^{(k)}_n$, although it would not be as direct as the bijection in Corollary~\ref{cor:PG}, which can be described easily at the level of paths.

Our final open question was formulated by Bousquet-M\'elou and Mishna in~\cite{BMM}:

\begin{problem}\label{problemBMM}
Find a direct (and involution-free) bijection between walks in $\O_{2m}$ ending on the $x$-axis and those ending on the diagonal.
\end{problem}

Whereas Corollary~\ref{cor:diagonal} proves Equation~\eqref{eq:Odiag} bijectively, it does not completely solve the above problem, partly because Gouyou-Beauchamps's proof of Equation~\eqref{eq:Oxaxis} involves negative signs and cancellations. Note also that our bijection $\bijw$ restricts to the identity on paths in the octant that end on the $x$-axis.
In terms of nested paths, Problem~\ref{problemBMM} translates into finding a direct bijection between
those pairs $(P,Q)\in\PP_{2m}$ with $h(P)=h(Q)$ and those with $h(Q)=0$. This question is reminiscent of the 
symmetry between top and bottom contacts of lattice paths between two fixed boundaries discussed in~\cite{EliRub}.

\subsection*{Acknowledgments}
The author thanks Mireille Bouquet-M\'elou, Ira Gessel, Mark Haiman, Christian Krattenthaler and Richard Stanley for useful discussions and for providing many relevant references.

\end{document}